\documentclass{amsart}[12]
\usepackage[utf8]{inputenc}
\usepackage{amsfonts, amssymb, amsthm, amsmath, mathtools, graphicx, hyperref, caption, comment, color}
\usepackage{comment,tikz,epstopdf}
\usepackage{graphics}

\theoremstyle{definition}

\newtheorem{theorem}{Theorem}[section]
\newtheorem{corollary}[theorem]{Corollary}
\newtheorem{lemma}[theorem]{Lemma}

\newtheorem{definition}[theorem]{\rm{Definition}}
\newtheorem{proposition}[theorem]{Proposition}

\DeclareMathOperator{\Aut}{Aut}

\title{Topological symmetry groups of the Generalized Petersen Graphs}
\author{A. \'Alvarez}
\author{E. Flapan}
\author{M. Hunnell}
\author{J. Hutchens}
\author{E. Lawrence}
\author{P. Lewis}
\author{C. Price}
\author{R. Vanderpool}

\thanks{The second author was supported in part by NSF Grant DMS-1607744}

\address{Department of Mathematics, Embry-Riddle Aeronautical University}
\address{Department of Mathematics, Pomona College}
\address{Department of Mathematics, Winston Salem State University}
\address{Department of Mathematics and Statistics, University of San Francisco}
\address{Department of Mathematics and Statistics, University of San Francisco}
\address{Department of Mathematics, Houston Baptist University}
\address{Department of Mathematics, Smith College}
\address{School of Interdisciplinary Arts and Sciences, University of Washington Tacoma}


\thanks{}

\begin{document}

\maketitle

\begin{abstract} The topological symmetry group $\mathrm{TSG}(\Gamma)$ of an embedded graph $\Gamma$ in $S^3$ is the subgroup of the automorphism group of the graph which is induced by homeomorphisms of $(S^3,\Gamma)$.  If we restrict to orientation preserving homeomorphisms then we obtain the orientation preserving topological symmetry group $\mathrm{TSG}_+(\Gamma)$.  In this paper, we determine all groups that can be $\mathrm{TSG}(\Gamma)$ or $\mathrm{TSG}_+(\Gamma)$ for some embedding $\Gamma$ of a generalized Petersen graph other than the exceptional graphs $P(12,5)$ and $P(24, 5)$.\end{abstract}

\section{Introduction}\label{introduction}

The field of spatial graph theory was developed in part to classify the symmetries of flexible molecules \cite{Simon} and in part as an extension of the study of the symmetries of knots and links \cite{CG}.  However, in contrast with the symmetries of a knot or link, the symmetries of a spatial graph can be understood in terms of the automorphism group of its underlying graph. In particular, we have the following definitions.

\begin{definition} Let $\gamma$ be an abstract graph and $\Gamma$ be an embedding of $\gamma$ in $S^3$.  The \emph{topological symmetry group of $\Gamma$}, denoted by $\mathrm{TSG}(\Gamma)$, is the subgroup of the automorphism group $\mathrm{Aut}(\gamma)$ induced by homeomorphisms of the pair $(S^3,\Gamma)$.  If we only allow orientation preserving homeomorphisms, we obtain the \emph{orientation preserving topological symmetry group}, $\mathrm{TSG}_+(\Gamma)$.  \end{definition}

\begin{definition} Let $\gamma$ be an abstract graph and let $G$ be a subgroup of $\mathrm{Aut}(\gamma)$ such that for some embedding $\Gamma$ of $\gamma$ in $S^3$, $G=\mathrm{TSG}(\Gamma)$ or $G=\mathrm{TSG}_+(\Gamma)$.  Then we say that the group $G$ is \emph{realizable} or \emph{positively realizable}, respectively, for $\gamma$.\end{definition}

 Many previous results have been obtained about realizable and positively realizable groups for specific graphs (see for example \cite{CF},\cite{FL},\cite{FMN},\cite{FMNY}).  In this paper, we classify the groups that are realizable and positively realizable for the family of \emph{generalized Petersen graphs}.  In particular, for $2k < n$, the generalized Petersen graph $P(n,k)$ is obtained from an $n$-gon with consecutive vertices $u_1,u_2,\dots,u_n$ and a star (possibly with more than one loop) with vertices $v_1,v_2,\dots,v_n$ and edges $\overline{ v_i v_{i+k}}$, by adding an edge $\overline{u_iv_i}$ for each $i\leq n$ (see the graph $P(6,2)$ and an embedding of $P(7,2)$ in Figure~\ref{P62}).  We call the edges $\overline{u_iu_{i+1}}$ \emph{outer edges}, the edges $\overline{v_iv_{i+k}}$ \emph{inner edges}, and the edges $\overline{u_iv_i}$ \emph{spokes}.  Note that when $k^2\equiv \pm1 \pmod{n}$, then $k$ and $n$ are relatively prime and hence the star formed by the inner edges of $P(n,k)$ is a single cycle.

\begin{figure}[h!]
\begin{center}
\centering\includegraphics[scale=0.3]{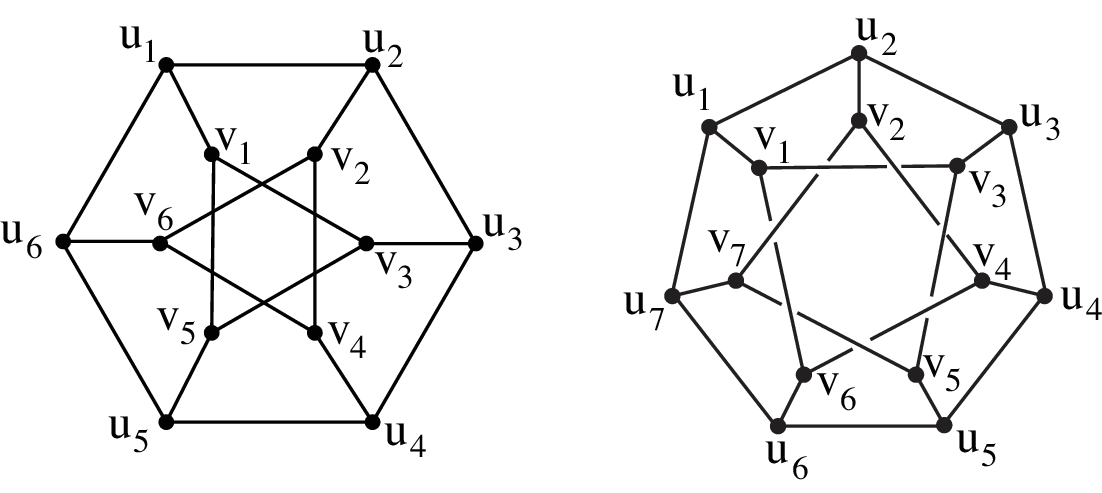}
\caption{The graph $P(6,2)$ and an embedding of $P(7,2)$.}
\label{P62}
\end{center}
\end{figure}

Let $B(n,k)$ denote the subgroup of $\mathrm{Aut}(P(n,k))$ which leaves the set of spokes $\{\overline{u_iv_i}| i\leq n\}$ invariant (either preserving or interchanging the inner and outer edges). Frucht, Graver, and Watkins \cite{FRU} proved that all but the seven exceptional pairs $(4,1)$, $(5,2)$, $(8,3)$, $(10,2)$, $(10,3)$, $(12,5)$, $(24,5)$  have $\mathrm{Aut}(P(n,k))=B(n,k)$.  Furthermore they showed that    
\[ B(n,k) = 
	\begin{cases}
	D_n, &k^2 \not\equiv \pm 1 \ \pmod{n} \\
	D_n \rtimes \mathbb{Z}_2, &k^2 \equiv 1  \ \pmod{n} \\
	\mathbb{Z}_n \rtimes \mathbb{Z}_4, &k^2 \equiv -1  \ \pmod{n}. \\
	\end{cases}
\]

We classify the realizable and positively realizable subgroups of $\mathrm{Aut}(P(n,k))$ for all non-exceptional pairs $(n,k)$ and for the exceptional pairs $(4,1)$, $(8,3)$, $(10,2)$, $(10,3)$.  For the exceptional pair $(5,2)$, the graph $P(5,2)$ is the \emph{Petersen graph} whose topological symmetry groups were determined by Chambers, Flapan, Heath, Lawrence, Thatcher, and Vanderpool \cite{CFLTV}.  Realizability and positive realizability for $P(12,5)$ and $P(24,5)$ will be considered in a subsequent paper. The exceptional graphs  $P(n,k)$ for $n\leq 10$ are illustrated in Figure~\ref{exceptional}.  

\begin{figure}[h!]
\begin{center}
\centering\includegraphics[scale=0.35]{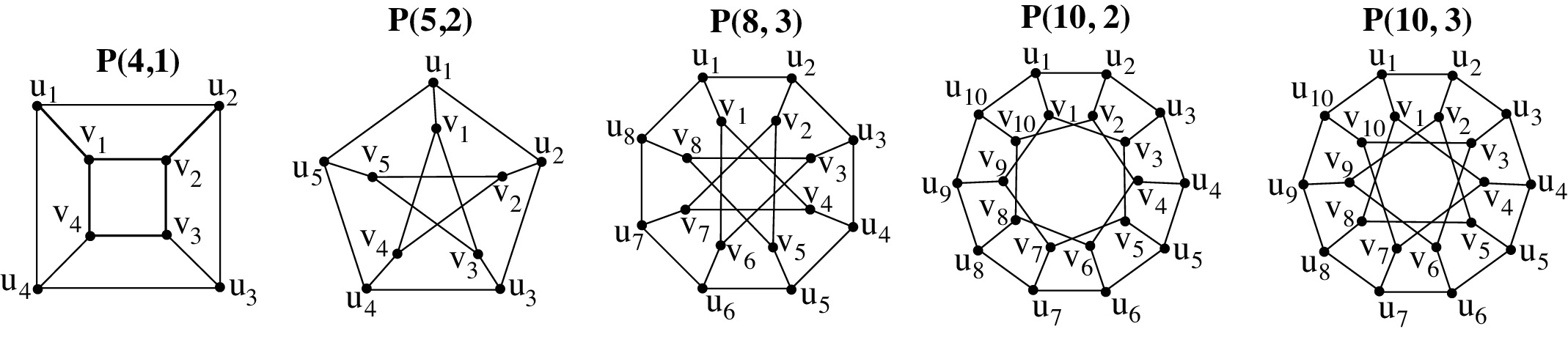}
\caption{The exceptional graphs $P(4,1)$, $P(5,2)$, $P(8,3)$, $P(10, 2)$, and $P(10,3)$.}
\label{exceptional}
\end{center}
\end{figure}

Our arguments make use of the above theorem of Frucht et al, as well as the theorems listed below.  However, for a large number of groups, we prove realizability or positive realizability by constructing intricate embeddings which we show have the required group of symmetries and not a larger group (the latter is often the difficult part).  These embeddings are described in the text, and augmented by detailed illustrations. Some readers will find these constructions to be   the most interesting part of the paper, since they can likely be generalized to other graphs.

If $G=\mathrm{TSG}_+(\Gamma)$ for some embedding $\Gamma$ of $P(n,k)$,  then we can add the same chiral invertible knot to every edge of $\Gamma$ to get an embedding $\Gamma'$ with $G=\mathrm{TSG}(\Gamma')$.  This means that if we can find an embedding of $P(n,k)$ such that $G=\mathrm{TSG}_+(\Gamma)$, then $G$ is both positively realizable and realizable for $P(n,k)$.  Note however, that $G$ may be realizable and not be positively realizable.  Also, by putting a distinct knot on each edge of $\Gamma$ we obtain an embedding $\Gamma''$ such that $\mathrm{TSG}(\Gamma'')=\mathrm{TSG}_+(\Gamma'')=\langle e\rangle$.  Thus the trivial group is realizable and positively realizable for every $P(n,k)$.  Hence we focus on  non-trivial groups.  

Since $P(n,k)$ is $3$-connected, we can use the following theorems.

\begin{theorem} [Automorphism Rigidity Theorem \cite{Flap}] \label{Rigid}
Let $G$ be a 3-connected graph. Suppose that an automorphism $\sigma$ of $G$ is realizable by a homeomorphism $h$ of some embedding of $G$ in $S^3$. Then $\sigma$ is realizable by a homeomorphism $f$ of finite order which is orientation reversing if and only if $h$ is orientation reversing.
\end{theorem}

The above theorem allows us to assume realizable automorphisms are induced by finite order homeomorphisms of $S^3$, and then use the following simplified version of a theorem of P. A. Smith~\cite{SMI}.  We let $\mathrm{fix}(h)$ denote the fixed point set of $h$.

\begin{theorem}[Smith Theory] Let $h$ be a finite order homeomorphism of $S^3$.  If $h$ is orientation preserving then $\mathrm{fix}(h)$ is either the empty set or $S^1$, and if $h$ is orientation reversing then $\mathrm{fix}(h)$ is either two points or $S^2$.
\end{theorem}

We now list several other results that will be used.  Note that $\mathrm{SO}(4)$ is the group of orientation preserving isometries of $S^3$ and $\mathrm{SO}(3)$ is the group of orientation preserving isometries of $S^2$.  The Group Rigidity Theorem below follows from Proposition~3 of \cite{FNPT} together with the Geometrization Theorem \cite{MOR}.  
\begin{theorem}[Group Rigidity Theorem]  \label{GroupRigid} Let $\gamma$ be a $3$-connected graph and $G\leq\mathrm{TSG}_+(\Lambda)$ for some embedding $\Lambda$ of $\gamma$ in $S^3$.  Then there is an embedding $\Gamma$ of $\gamma$ in $S^3$ such that $G\leq \mathrm{TSG}_+(\Gamma)$ and $\mathrm{TSG}_+(\Gamma)$ is induced by an isomorphic finite subgroup of $\mathrm{SO}(4)$. 
\end{theorem}

\begin{theorem}[Involution Theorem  \cite{FNT}]\label{involution} Let $G\leq SO(4)$ such that for every involution $g\in G$, we have $\mathrm{fix}(g)\cong S^1$ and no $h\in G$  with $g\neq h$ has $\mathrm{fix}(h)=\mathrm{fix}(g)$.  Then $G$ is a subgroup of $D_m\times D_m$ for some odd $m$ or is a finite subgroup of $\mathrm{SO}(3)$. \end{theorem}

 \begin{theorem} (Subgroup Theorem \cite{FMN}) \label{SubgroupTheorem}
Let $\Gamma$ be a $3$-connected graph embedded in $S^3$. Suppose $\Gamma$ contains an edge whose vertices are not both fixed by a non-trivial element of $\mathrm{TSG}_+(\Gamma)$.  Then for every $H\leq \mathrm{TSG}_+(\Gamma)$ there is an embedding $\Gamma'$ of $\Gamma$ in $S^3$ with $H=\mathrm{TSG}_+(\Gamma')$.
\end{theorem}

\begin{lemma}\label{edgeembeddinglemma}
(Edge Embedding Lemma \cite{FMN}) 
Let $G$ be a finite group of diffeomorphisms of $S^3$ and let $\gamma$ be a graph whose vertices are embedded in $S^3$ as a set $W$ such that $G$ induces a faithful action on $W$ (i.e., no non-trivial element of $G$ induces the identity on $W$). Let $Y$ denote the union of the fixed point sets of all of the non-trivial elements of $G$.  Suppose that the vertices in $W$ satisfy the following:

\begin{enumerate}
\item
 If a pair of adjacent vertices is pointwise fixed by non-trivial elements $h$, $g\in G$, then $\mathrm{fix}(h)=\mathrm{fix}(g)$.
\item
No pair of adjacent vertices is interchanged by an element of $G$.
\item Any pair of adjacent vertices that is pointwise fixed by a non-trivial $g\in G$ bounds an arc in $\mathrm{fix}(g)$ whose interior is disjoint from $W\cup (Y-\mathrm{fix}(g))$.
\item Every pair of adjacent vertices which are not both fixed by a non-trivial element of $G$ is contained in a single component of $S^3-Y$.
\end{enumerate}
Then there is an embedding $\Gamma$ of the graph $\gamma$ in $S^3$ such that $\Gamma$ is setwise invariant under $G$.  
\end{lemma}

We also use the website https://people.maths.bris.ac.uk/~matyd/GroupNames/ and the software Sage  to identify the isomorphism classes of automorphism groups and their subgroups.

\section{Realizability of $D_n$ and its subgroups for all $P(n,k)$}

\begin{theorem} \label{Thm:Gen_Dn}  For every pair $(n,k)$, the dihedral group $D_n$ and all of its subgroups are positively realizable and hence realizable for $P(n,k)$. 
\end{theorem}

\begin{proof}  We construct an embedding of $P(n,k)$ as follows. Let $0<r<R$. We  begin with planar circles $C$ and $c$, centered at the origin, of radius $R$ and $r$, respectively.  Let $u_1$ and $v_1$ be the points of intersection between a planar ray from the origin and $C$ and $c$, respectively.  Next let $e_1$ denote the straight line segment between $u_1$ and $v_1$.  Then $e_1$ lies on the ray. Let $f$ denote a rotation of $S^3$ of $\frac{2\pi}{n}$ about an axis through the origin which is perpendicular to the plane of the projection.  Then for each $i$ let $u_i=f^i(u_1)$, $v_i=f^i(v_1)$, $e_i=f^i(e_1)$.  Thus we have embedded the spokes $\overline{u_iv_i}$ as the $e_i$.  Next we embed the outer edges $\overline{u_iu_{i+1}}$ as straight line segments.  We will embed the inner edges below.

Let $g$ denote a rotation of $\pi$ about a planar axis that contains the spoke $\overline{u_nv_n}$ if $n$ is odd and contains the two opposite spokes $\overline{u_nv_n}$ and $\overline{u_{\frac{n}{2}}v_{\frac{n}{2}}}$ if $n$ is even.  Then the isometry group $D_n=\langle f, g\rangle$ leaves the $n$-gon of outer edges and the set of spokes invariant, inducing an isomorphic group on the embedded vertices of $P(n,k)$.  Now we embed the inner edges $\overline{v_iv_{i+k}}$ so that the collection of edges $\{\overline{v_iv_{i+k}}| \medspace i
\leq n\}$ is pairwise disjoint and setwise invariant under $D_n$.  We can do this, for example, by making the crossings along each inner edge alternate, as illustrated for the embedding of $P(7,2)$ on the right of Figure \ref{P62}.   Let $\Gamma'$ denote this embedding of $P(n,k)$.    Thus $\Gamma'$ is invariant under $\langle f, g\rangle$, and hence $D_n\leq \mathrm{TSG}_+(\Gamma')$.


We obtain the embedding $\Gamma$ from $\Gamma'$ by adding the invertible knot $4_1$ to each outer edge.  Then $\Gamma$ is invariant under $\langle f, g\rangle$, and hence $D_n\leq \mathrm{TSG}_+(\Gamma)$.  Let $L$ denote the $n$-gon of outer edges of $\Gamma'$.  Then $L$ is the only $n$-gon containing $n$ copies of the knot $4_1$, and hence any element of $\mathrm{TSG}_+(\Gamma)$ must take $L$ to itself.  Since no non-trivial automorphism of $P(n,k)$ fixes every $u_i$, this implies that $D_n\leq \mathrm{TSG}_+(\Gamma)\leq \mathrm{TSG}_+(L)$.  On the other hand, because the automorphism group of an $n$-gon is $D_n$, we have $\mathrm{TSG}_+(L)\leq D_n$, and hence $D_n=\mathrm{TSG}_+(\Gamma)$.

No outer edge can be pointwise fixed by a non-trivial element of $\mathrm{TSG}_+(\Gamma)$ since that would pointwise fix $L$ and hence all of $\Gamma$.  Thus by Theorem \ref{SubgroupTheorem}, every subgroup of $D_n$ is positively realizable for $P(n,k)$.  \end{proof}

By \cite{FRU} we know that for $k^2\not\equiv\pm 1\pmod{n}$ where $(n,k)$ is non-exceptional, $\mathrm{Aut}(P(n,k))=B(n,k)=D_n$.  Thus we have the following.

\begin{corollary} Let $k^2\not\equiv\pm 1\pmod{n}$ where $(n,k)$ is non-exceptional.  Then $\Aut(P(n,k))$ and all of its subgroups are positively realizable and hence realizable for $P(n,k)$. \end{corollary}

\section{ The case $k^2\equiv1\pmod{n}$}\label{ksquareeq1modn}

We use the following embedding $\Lambda$ of $P(n,k)$ in this and the next section.  Let $U$ and $V$ denote the cores of complementary isometric solid tori in $S^3$.  Hence $U$ and $V$ are geodesic circles, and for every point on one of these cores, its antipodal point is on the same core.  Now, for every $u\in U$ and $v\in V$, there is a unique shortest geodesic $e$ joining $u$ and $v$, and the length of $e$ is less than $\pi$.  Since $U$ and $V$ are geodesic circles, it follows that the interior of any such $e$ must be disjoint from $U\cup V$.  

Suppose that $k^2 \equiv \pm1 \pmod{n}$.  Then $n$ and $k$ are relatively prime, and hence the inner edges of $P(n,k)$ form a single loop.  The embedding $\Lambda$ is obtained as follows.   We let $e_1$ be a geodesic of minimal length between $U$ and $V$, and let $u_1$ and $v_1$ be its endpoints on $U$ and $V$ respectively.  Let $f$ be a glide rotation which rotates $U$ by $\frac{2\pi}{n}$ while rotating $V$ by $k\big(\frac{2\pi}{n}\big)$.  Then for each $i$, define  $e_{i+1}=f^i(e_1)$, and define the endpoint of $e_{i+1}$ on $U$ to be $u_{i+1}$ and the endpoint on $V$ to be $v_{i+1}$.  Thus $f(e_i)=e_{i+1}$, $f(u_i)=u_{i+1}$, and $f(v_i)=v_{i+1}$.  

Since $f$ rotates $U$ by $\frac{2\pi}{n}$, the points $u_i$ and $u_{i+1}$ are consecutive on $U$.  Also, since $f$ rotates $V$ by $k\big(\frac{2\pi}{n}\big)$, we know that $f^k$ rotates $V$ by $k^2\big(\frac{2\pi}{n}\big)\equiv\frac{2\pi}{n}\pmod{n}$.  Since $f^k(v_i)=v_{i+k}$, this implies that  $v_i$ and $v_{i+k}$ are consecutive on $V$. Thus we define the edges $\overline{u_iu_{i+1}}$ and $\overline{v_iv_{i+k}}$ to be minimal arcs on $U$ and $V$, respectively.  It follows that the interiors of the $e_i$ are disjoint from $U$ and $V$.  Furthermore, if we consider the open book decomposition of $S^3$ whose spine is $U$, then $f$ rotates the pages of this decomposition around $U$ and each $e_i$ is on a distinct page.  Hence the $e_i$ are pairwise disjoint.  This gives us an embedding $\Lambda$ of $P(n,k)$.

 Now we orient $U$ and $V$ so that $\overrightarrow{u_iu_{i+1}}$ and $\overrightarrow{v_iv_{i+k}}$ are oriented positively.

\begin{theorem}\label{k2=1}
Let  $k^2 \equiv 1 \pmod{n}$.  Then $B(n,k)=D_n \rtimes \mathbb{Z}_2$ and all of its subgroups are positively realizable and hence realizable for $P(n,k)$.
\end{theorem}

\begin{proof}  We start with the embedding $\Lambda$ of $P(n,k)$ described above and we add a $4_1$ knot to each edge of $U$ and $V$ to get an embedding $\Lambda'$ of $P(n,k)$.  Recall that $f$ is a glide rotation which rotates $U$ by $\frac{2\pi}{n}$ while rotating $V$ by $k\big(\frac{2\pi}{n}\big)$ both in the positive direction.  Let $g$ be a rotation of $S^3$ by $\pi$ around an axis that contains $e_1=\overline{u_1v_1}$.  Then $g$ takes $U$ and $V$ to themselves, reversing their orientations.  Finally, let $h$ be an order $2$ rotation of $S^3$ interchanging the positively oriented $U$ with the positively oriented $V$, taking $u_n$ to $v_n$. 

In order to show that  $f$, $g$, and $h$ take $\Lambda'$ to itself, we need to know that they take edges to edges.  First, by definition we know that $f(e_i)=e_{i+1}$.   Also, for every $i$, we have $g(v_{1-ki})=v_{1+ki}$, and hence $g(v_{1-i})=g( v_{1-k^2i})=v_{1+k^2i}=v_{1+i}$.  Since we also have $g(u_{1-i})=u_{1+i}$ and we know that the $e_i$ are minimal length geodesics, it follows that $g(e_{1-i})=e_{i+1}$.  Finally, since $h$ interchanges $u_n$ and $v_n$ preserving orientation, we have $h(u_{i})=v_{ki}$, and hence $h(u_{ki})=v_{k^2i}=v_{i}$.  Thus for every $i$, we have $h(e_{i})=e_{ki}$.  It follows that $f$, $g$, and $h$ take $\Lambda'$ to itself.

Now $f$ induces a rotation of order $n$ on $U$ and $V$, $g$ turns $U$ and $V$ over, and $h$ is an involution interchanging $U$ and $V$.  Thus $\langle f, g, h \rangle$ induces $D_n \rtimes \mathbb{Z}_2$ on $\Lambda'$, and hence $D_n \rtimes \mathbb{Z}_2\leq \mathrm{TSG}_+(\Lambda')$.  However, since every element of $\mathrm{TSG}_+(\Lambda)$ takes the set of edges $e_i$ to themselves, it follows that $\mathrm{TSG}_+(\Lambda')\leq B(n,k)$.  Now by \cite{FRU}, $B(n,k)=D_n \rtimes \mathbb{Z}_2$ and hence $\mathrm{TSG}_+(\Lambda')=D_n \rtimes \mathbb{Z}_2$.

Finally, observe that no edge of $U$ or $V$ is pointwise fixed by any non-trivial element of $\mathrm{TSG}_+(\Lambda')$.  Thus $\Lambda'$ satisfies the hypothesis of Theorem \ref{SubgroupTheorem}, and hence all subgroups of $ D_n \rtimes \mathbb{Z}_2$ are positively realizable for $P(n,k)$.  \end{proof}

If $(n,k)$ is non-exceptional then $\mathrm{Aut}(P(n,k))=B(n,k)$, implying the following.

\begin{corollary} Let  $k^2 \equiv 1 \pmod{n}$ where $(n,k)$ is non-exceptional.  Then $\Aut(P(n,k))$ and all of its subgroups are positively realizable and hence realizable for $P(n,k)$. \end{corollary}

\section{The case $k^2 \equiv -1\pmod{n}$ }\label{ksquareequal-1modn}\label{k^2=-1}

Since $k^2 \equiv -1\pmod{n}$, we know that $k$ and $n$ are relatively prime, and hence the inner edges of $P(n,k)$ form a single cycle.  Furthermore, if $n$ were divisible by $4$, then $k$ would be odd, and hence $k^2=(2m+1)^2=4m^2+4m+1$ for some $m$.   But this implies that $4m^2+4m+1\equiv -1\pmod{n}$, and thus $4m^2+4m+2=nr$ for some $r$, which is impossible since $n$ is divisible by $4$.  Therefore, $n$ cannot be divisible by $4$.

According to \cite{FRU},  $B(n,k)=\mathbb{Z}_n \rtimes \mathbb{Z}_4= \langle \rho, \alpha \ | \ \rho^n = \alpha^4 = \mathrm{id}, \alpha \rho \alpha^{-1} = \rho^k \rangle$, where $\alpha(u_i) = v_{ki}$, $\alpha(v_i) = u_{ki}$, $\rho(u_i) = u_{i+1}$, and $\rho(v_i) = v_{i+1}$.   Observe that every element of $B(n,k)$ can be expressed as $ \rho^m\alpha^r$ for some $0 \leq m <n$ and $0 \leq r < 4$.

\begin{lemma}  \label{NoOrder4}Let $k^2 \equiv -1\pmod{n}$.  Then an element of $B(n,k)$ has order $4$ if and only if it can be expressed as $\rho^m\alpha^{\pm1}$.  Furthermore,  if $n$ is odd, then no order $4$ element of $B(n,k)$ is positively realizable. \end{lemma}

\begin{proof} Since $\alpha$ interchanges the inner and outer cycles of $P(n,k)$ while $\rho$ induces an order $n$ rotation of both cycles, for any $m<n$, $\rho^m$ rotates both cycles while $\rho^m\alpha^{\pm1}$ interchanges them. Furthermore, $\alpha^2(u_i)=\alpha(v_{ki})=u_{k^2i}=u_{-i}$ and $\alpha^2(v_i)=v_{-i}$.  Thus $\alpha^2$ is an involution which turns over both the inner and outer cycles, and hence so is any $\rho^m\alpha^2$.  It follows that every element of the form $\rho^m\alpha^{\pm1}$ has order $4$.  Also, since $n$ is not divisible by $4$, $D_n=\langle \rho,\alpha^2\rangle$ has no elements of order $4$.  Thus an element of $B(n,k)$ has order $4$ if and only if it can be expressed as $\rho^m\alpha^{\pm1}$.

Suppose that an order $4$ element $\beta$ of $B(n,k)$ is induced by an orientation preserving homeomorphism of some embedding of $P(n,k)$ in $S^3$. By the above paragraph, $\beta$ has the form $\rho^m\alpha^{\pm1}$.  Now by Theorem~\ref{Rigid}, there is an embedding $\Gamma'$ of $P(n,k)$ and a finite order orientation preserving homeomorphism $h$ of $(S^3,\Gamma')$ which also induces $\beta$. Now $h^2$ turns over the inner and outer cycles of $\Gamma$, and hence fixes two points on each of these cycles.  Thus by Smith Theory, $\mathrm{fix}(h^2)=S^1$.  Since $h$ is orientation preserving, if $h$ is not fixed point free, then we also have $\mathrm{fix}(h)=S^1$.  Since $\mathrm{fix}(h)\subseteq\mathrm{fix}(h^2)$, we have $\mathrm{fix}(h)=\mathrm{fix}(h^2)$.  This implies that $h$ fixes two points on each of the cycles, which is impossible since $\rho^m\alpha^{\pm1}$ interchanges the two cycles.  Thus $h$ must be fixed point free. 

Suppose that $n$ is odd. Since $\beta\in B(n,k)$, $h$ must leave at least one spoke $e_i$ setwise invariant.  But this implies that $h$ fixes the midpoint of $e_i$, which contradicts the above paragraph.  Thus no order $4$ element of $B(n,k)$ is positively realizable. \end{proof}

Using the above lemma we see as follows that $D_4$ is not a subgroup of $B(n,k)$.  In particular, if $D_4$ were a subgroup, then by the lemma it would be generated by an element $\rho^m\alpha$ of order $4$ and an element $\rho^r\alpha^2$ of order $2$.  But $\rho^m\alpha\rho^r\alpha^2=\rho^{m+nk}\alpha^3$ has order $4$, whereas in $D_4$ the product of a pair of generators has order $2$.

\begin{theorem}\label{order4}  Let $k^2 \equiv -1\pmod{n}$ and $H\leq B(n,k)=\mathbb{Z}_n \rtimes \mathbb{Z}_4$.  If $n$ is odd, then $H$ is positively realizable for $P(n,k)$ if and only if $H\leq D_n$.  If $n$ is even, then $H$ is positively realizable for $P(n,k)$ if and only if either $H\leq D_n$ or $H=\mathbb{Z}_4$.
\end{theorem}

\begin{proof} By Theorem~\ref{Thm:Gen_Dn}, $D_n$ and all of its subgroups are positively realizable for $P(n,k)$.  If $H$ is not contained in $D_n$, then $H$ must contain an element of order $4$.  When $n$ is odd, no order $4$ element of $B(n,k)$ is positively realizable by Lemma~\ref{NoOrder4}, and hence $H$ is not positively realizable.

Thus we assume that $n$ is even.  Since $n$ is not divisible by $4$, we have $n\equiv{2}\pmod{4}$.  In order to construct an embedding of $P(n,k)$ which positively realizes $\mathbb{Z}_4\leq B(n,k)$, we consider the glide rotation $h$ which rotates a standardly embedded solid torus meridionally by $\frac{\pi}{2}$ while rotating it longitudinally by $\pi$.

Let $U$ be a meridian of the solid torus with $n-2$ evenly spaced vertices labeled $u_1,\dots, u_{\frac{n}{2}-1}, u_{\frac{n}{2}+1},\dots, u_{n-1}$. Let $V=h(U)$ with vertices $v_k, v_{2k},\dots, v_{\frac{n}{2}-k}, v_{\frac{n}{2}}+k,\dots,v_{n-k}$ such that each $v_j$ is $h(u_i)$ for some $i$.   Let $D_U$ and $D_V$ denote disjoint meridional disks bounded by $U$ and $V$ respectively, such that $h(D_U)=D_V$.  Let $C=\mathrm{fix}(h^2)$, and let $x$ denote the midpoint of an arc $A$ of $C-(D_U\cup D_V)$. Let $u_n$ be a vertex on the arc of $A-\{x\}$ with one endpoint on $D_U$, and let $v_n$ be a vertex on the arc of $A-\{x\}$ with one endpoint on $D_V$.  Finally, let $v_{\frac{n}{2}}=h(u_n)$ and $u_{\frac{n}{2}}=h(v_n)$.  This gives us an embedded set of vertices $W$.  

For example, Figure~\ref{Fig:Z4} illustrates the embedded vertices of $P(10,3)$ except for $v_{5}$ and $u_{5}$, which are in the front of the solid torus.  Note that the core is illustrated in green, and $h$ takes the pair of blue arcs on $U$ to the pair of blue arcs on $V$.

  \begin{figure}[h]
\begin{center}
\centering\includegraphics[height=0.1\textheight]{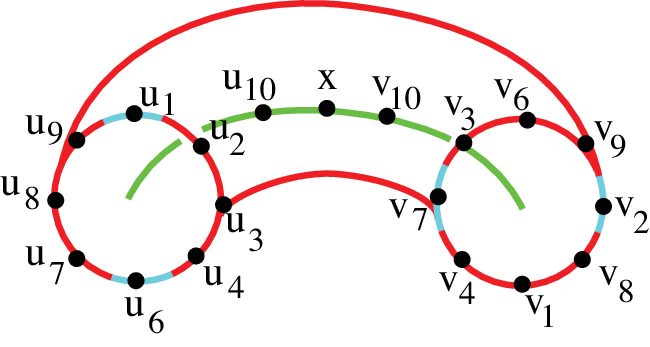}
\caption{The embedded vertices of $P(10,3)$, except for $u_5$ and $v_5$ which are on the (green) core in the front of the solid torus.}
\label{Fig:Z4}
\end{center}
\end{figure}

Let $G=\langle h\rangle=\mathbb{Z}_4$.  Then $G$ induces a faithful action $H\leq B(n,k)$ of $W$ such that no pair of adjacent vertices are fixed by a non-trivial element of $H$.  The pairs of adjacent vertices  $u_n$ and $v_n$ and $u_{\frac{n}{2}}$ and $v_{\frac{n}{2}}$ are the only ones which are fixed by $h^2$.  But they are not fixed by any other non-trivial element of $G$.  Also these pairs each bound an arc in $C$ which is disjoint from the other vertices.  Hence we can apply Lemma~\ref{edgeembeddinglemma} to embed the edges of $P(n,k)$ such that the resulting embedded graph is setwise invariant under $G$.

Next assume that a positively realizable group $H\leq B(n,k)$ contains $\mathbb{Z}_4$ as a proper subgroup.   Then by Theorem~\ref{GroupRigid}, for some embedding $\Gamma$ of $P(n,k)$ in $S^3$, the group $H$ is induced on $\Gamma$ by an isomorphic group $H'\leq \mathrm{SO}(4)$.  Every involution in $H$ has the form $\rho^m\alpha^2$ and hence turns over both the inner and the outer cycle of $\Gamma$.  Thus every involution in $H'$ fixes two points on each of these cycles.  Furthermore, every non-trivial orientation preserving isometry of $(S^3,\Gamma)$ which fixes two points on each cycle must be an involution and if two such isometries fix the same points on these cycles then the isometries are equal.  

Thus we can apply Theorem~\ref{involution} to conclude that $H'$ is either a subgroup of $D_m\times D_m$ for some odd $m$ or a finite subgroup of $\mathrm{SO}(3)$.
However, since $H\leq B(n,k)$ and contains $\mathbb{Z}_4$ as a proper subgroup, $H$ has the form $\mathbb{Z}_r\rtimes \mathbb{Z}_4$.  But this is impossible since $H'\cong H$.  It follows that no subgroup of $B(n,k)$ containing $\mathbb{Z}_4$ as a  proper subgroup is positively realizable for $P(n,k)$.  \end{proof}

\begin{theorem} \label{TSG(k2=-1)}  Let $k^2 \equiv-1\pmod{n}$. Then $B(n,k)=\mathbb{Z}_n \rtimes \mathbb{Z}_4$ and all of its subgroups are realizable for $P(n,k)$.
\end{theorem}

\begin{proof}  Let $\Lambda$ denote the embedding of $P(n,k)$ from the beginning of Section~\ref{ksquareeq1modn}.  
Let $h$ be an order $2$ rotation of $S^3$ interchanging the positively oriented $U$ with the positively oriented $V$ such that it interchanges $u_i$ and $v_{ki}$.  Now $h(u_{ki})=v_{k^2i}=v_{-i}$, and hence $h(v_i)=u_{-ki}$.  Also, let $R$ be a reflection through a sphere containing the circle $V$ which leaves $U$ setwise invariant, fixing $u_n$ and its antipodal point on $U$ (which is a vertex if $n$ is even).   Thus for every $i$, we have $R(u_{i})=u_{-i}$ and $R(v_i)=v_{i}$.  While both $h$ and $R$ take edges contained in $U$ and $V$ to other such edges, neither $h$ nor $R$ takes edges of the form $\overline{u_iv_i}$ to other edges.  Thus we are instead interested in the orientation reversing isometry $Rh$.  Observe that $Rh(u_i)=R(v_{ki})=v_{ki}$ and $Rh(v_i)=R(u_{-ki})=u_{ki}$.  Thus for each spoke $e_i$, we have $Rh(e_i)=e_{ki}$.  It follows that $Rh$ takes $\Lambda$ to itself interchanging $U$ and $V$ and $(Rh)^2$ turns over $U$ and $V$, fixing $u_n$ and $v_n$.  In particular, $Rh$ induces $\alpha$ on $\Lambda$.

Next, let $f$ be a glide rotation which rotates $U$ by $\frac{2\pi}{n}$ in the positive direction while rotating $V$ by $\frac{2k\pi}{n}$ in the negative direction, so that $f(u_i)=u_{i+1}$ and $f(v_{i})=v_{i-k^2}=v_{i+1}$.  Thus $f(e_i)=e_{i+1}$.  It follows that $f$ takes $\Lambda$ to itself inducing $\rho$.  Hence $B(n,k)=\mathbb{Z}_n \rtimes \mathbb{Z}_4=\langle \alpha, \rho\rangle\leq \mathrm{TSG}(\Lambda)$.  Now we add the  knot $4_1$ to each of the spokes $e_i$ of $\Lambda$ to get an embedding $\Lambda'$ such that an automorphism is in $B(n,k)$ if and only if it is in $\mathrm{TSG}(\Lambda')$.   Thus $\mathbb{Z}_n \rtimes \mathbb{Z}_4$ is realizable for $P(n,k)$.

  We show as follows that all proper subgroups of $\mathbb{Z}_n \rtimes \mathbb{Z}_4$ are also realizable.  First observe that every subgroup of $B(n,k)$ with no element of order $4$ is a subgroup of $D_n$.  Thus by Lemma~\ref{Thm:Gen_Dn} it is positively realizable, and hence it is also realizable.  
  
  Let $G$ denote a proper subgroup of $B(n,k)$ which contains an element of order $4$.  Since $n$ is not divisible by $4$, $G$ must be isomorphic to $\mathbb{Z}_r \rtimes \mathbb{Z}_4$, for some $r,m>1$ with $rm=n$.   Now starting with $\Lambda'$, we add the achiral invertible knot $6_3$ to the edge $\overline{u_nu_1}$ and all of the edges in its orbit under $\langle \rho^m, \alpha\rangle$ to obtain an embedding $\Lambda''$.  Since $\alpha$ interchanges $U$ and $V$, and $\alpha^2$ flips $U$ and $V$ over fixing $u_n$ and $v_n$, $\Lambda''$ contains the $6_3$ knot on $r$ pairs of adjacent edges on $U$ and $r$ pairs of adjacent edges on $V$, but not on any other edges.  Thus $\mathbb{Z}_r \rtimes \mathbb{Z}_4=\langle \rho^m, \alpha\rangle\leq \mathrm{TSG}(\Lambda'')\leq \mathbb{Z}_r \rtimes \mathbb{Z}_4$.  Hence every subgroup of $\mathbb{Z}_n \rtimes \mathbb{Z}_4$ is realizable.  \end{proof}

If $(n,k)$ is non-exceptional then $\mathrm{Aut}(P(n,k))=B(n,k)$, implying the following.

\begin{corollary}Let $k^2 \equiv-1\pmod{n}$ and suppose that $(n,k)$ is non-exceptional. Then $\mathrm{Aut}(P(n,k))$ and all of its subgroups are realizable for $P(n,k)$.  If $n$ is odd then the only groups which are positively realizable for $P(n,k)$ are $D_n$ and its subgroups, and if $n$ is even, then the only groups which are positively realizable for $P(n,k)$ are $\mathbb{Z}_4$ and $D_n$ and its subgroups.
\end{corollary}

\section{The exceptional case $P(4,1)$}\label{(4,1)}
We know from \cite{FRU1} that $\Aut((P(4,1)) = S_4 \times \mathbb{Z}_2$.  

\begin{proposition}\label{S4Z2}
Let $\Gamma$ be the embedding of $P(4,1)$ in $S^3$ as the $1$-skeleton of a cube.  Then $\mathrm{TSG}(\Gamma) = \mathrm{TSG}_+(\Gamma) = S_4 \times \mathbb{Z}_2$.
\end{proposition}

\begin{proof} The group $\Aut((P(4,1)) = S_4 \times \mathbb{Z}_2$ is induced on $\Gamma$ by the rotations and reflections of a cube.  To see that $S_4 \times \mathbb{Z}_2$  is also positively realized by $\Gamma$, we  flatten out the cube so it is a small square inside of a big square with edges between them.  Then the automorphisms induced by reflections of the cube can also be induced by turning the flattened cube over.  It follows that $\mathrm{TSG}(\Gamma) =\mathrm{TSG}_+(\Gamma) = S_4 \times \mathbb{Z}_2$.\end{proof}

 The non-trivial automorphisms of $P(4,1)$ which pointwise fix an edge, interchange two pairs of non-adjacent vertices and pointwise fix two non-adjacent edges.

\begin{proposition} \label{S4} $S_4$ and all of its subgroups are positively realizable and hence realizable for $P(4,1)$.   \end{proposition}

  \begin{figure}[h]
\begin{center}
\centering\includegraphics[height=0.28\textheight]{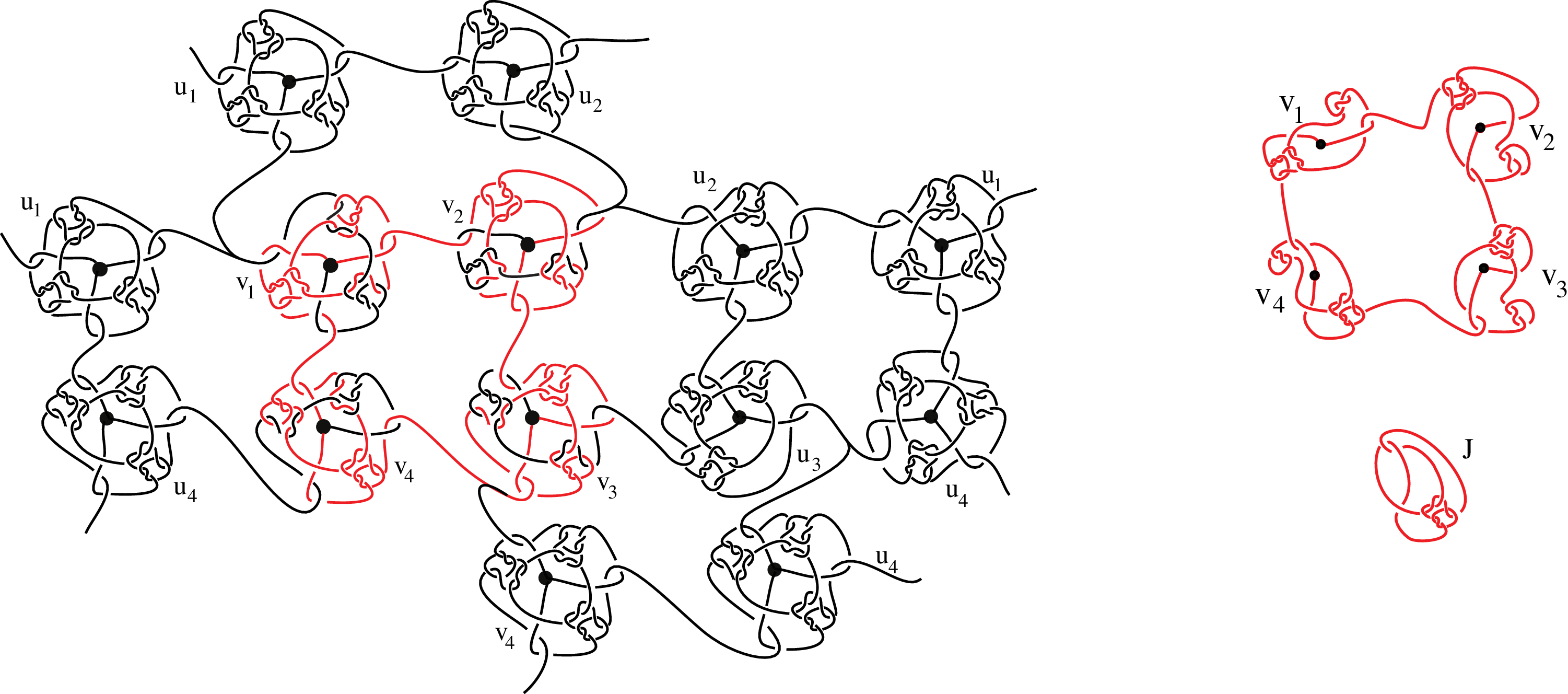}
\caption{The embedding $\Gamma$ of $P(4,1)$ with $\mathrm{TSG}_+(\Gamma) = S_4$ is obtained from the left image by identifying vertices with the same labels and pairs of adjacent branched edges. On the right, we see that $\overline{v_1v_2v_3v_4}$ is the connected sum of four copies of $J$ and four trefoils.}
\label{Fig:Woven}
\end{center}
\end{figure}

 \begin{proof} Let $\Gamma$ be an embedding of $P(4,1)$ as the skeleton of a cube with tangling around the vertices as indicated in the unfolded projection in Figure~\ref{Fig:Woven}.  On the right, we illustrate the knot $\overline{v_1v_2v_3v_4}$.

  Observe that the rotations of a solid cube leave $\Gamma$ invariant, inducing the automorphisms $\rho= (u_1u_2u_3u_4)(v_1v_2v_3v_4)$ and $\sigma = (u_1v_2v_4)(u_2v_3u_4)$ which generate $S_4$.  Thus $S_4=\langle \rho, \sigma \rangle \leq \mathrm{TSG}_+(\Gamma)$.  Let $\delta = (u_1u_3)(v_1v_3)$ and note that  $ \Aut((P(4,1)) =S_4 \times \mathbb{Z}_2=  \langle \rho, \sigma,\delta \rangle$ has $S_4=\langle \rho, \sigma \rangle$ as a maximal subgroup.  We show below that $\delta\not\in\mathrm{TSG}_+(\Gamma)$.

 Suppose that $\delta$ is induced on $\Gamma$ by a homeomorphism $h$ of $S^3$.  Because $\Gamma$ contains a $3_1$ knot but not its mirror image, $h$ must be orientation preserving.  Now $h$ induces the automorphism $(v_1v_3)$ of the knot $\overline{v_1v_2v_3v_4}$, which is a connected sum of four copies of $J$ and four $3_1$ knots.  It follows that $h$ interchanges two copies of $J$ while flipping over the other two copies of $J$.  But by applying the machinery of Bonahon and Siebenmann for algebraic knots \cite{BS} to $J$, we see that $J$ is non-invertible.  Thus $h$ cannot exist.  Hence $\delta\not\in\mathrm{TSG}_+(\Gamma)$, and therefore $\mathrm{TSG}_+(\Gamma) = S_4$.   
 
Since none of the non-trivial elements of $\langle \rho, \sigma \rangle=\mathrm{TSG}_+(\Gamma) = S_4$ pointwise fix an edge, we can use Theorem \ref{SubgroupTheorem} to positively realize all of the subgroups of $S_4$.   \end{proof}

\begin{proposition}
$ A_4 \times \mathbb{Z}_2$ is positively realizable and hence realizable for $P(4,1)$.
 \end{proposition}

\begin{proof} Let $\Gamma$ be the embedding of $P(4,1)$ whose unfolded projection on the faces of a cube is illustrated on the left in Figure~\ref{UnfoldedCube}, with the projection on a single face illustrated on the right.  
  \begin{figure}[h!]
\begin{center}
\centering\includegraphics[height=0.17\textheight]{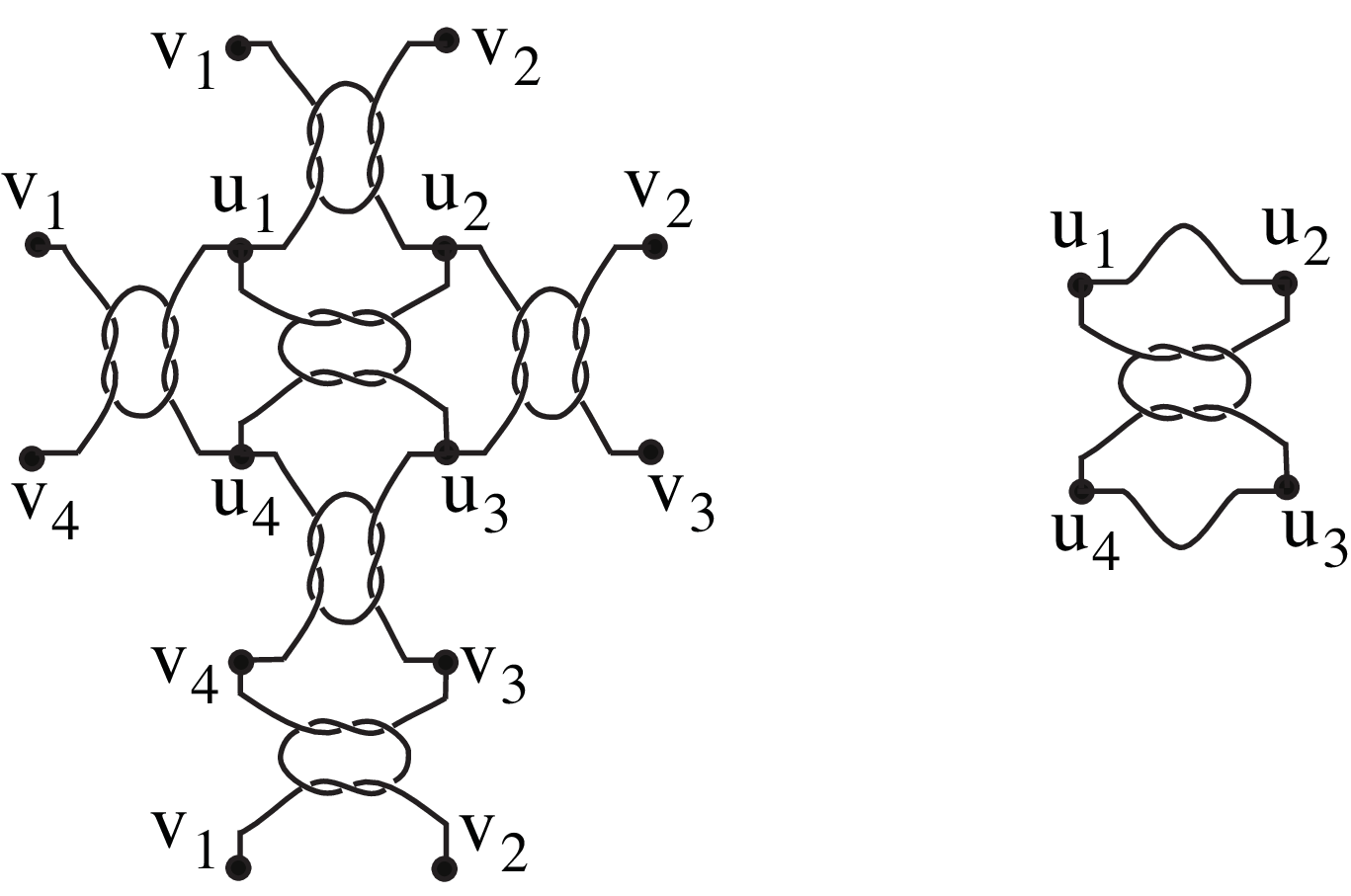}
\caption{The embedding $\Gamma$ of $P(4,1)$ with $\mathrm{TSG}_+(\Gamma) =A_4 \times \mathbb{Z}_2$ is obtained from the projection on the left by identifying vertices with the same labels.  On the right, we have the projection on a single face.}
\label{UnfoldedCube}
\end{center}
\end{figure}

\noindent Let 
$\alpha=(u_2u_4v_1)(u_3v_4v_2)$, $\beta=(u_1u_2u_3u_4)(v_1v_2v_3v_4)$, $\delta=(u_1v_1)(u_2v_2)(u_3v_3)(u_4v_4).$ 
Then $\langle \alpha, \beta,\delta\rangle=S_4\times \mathbb{Z}_2$, and has $\langle \alpha, \beta^2,\delta\rangle=A_4\times \mathbb{Z}_2$ as a maximal subgroup.  Observe that $\alpha$ is induced on $\Gamma$ by a rotation of $\frac{2\pi}{3}$ around an axis through vertices $u_1$ and $v_3$, and $\delta$ is induced on $\Gamma$ by a rotation around an equator of the cube that is parallel to $\overline{u_1u_2u_3u_4}$ and $\overline{v_1v_2v_3v_4}$.  Also, $\beta^2$ is induced by a rotation by $\pi$ around an axis through center of the faces containing $\overline{u_1u_2u_3u_4}$ and $\overline{v_1v_2v_3v_4}$.  However, the knot $\overline{u_1u_2u_3u_4}$ is the connected sum of two trefoils and hence does not have an order $4$ symmetry (see the right side of Figure~\ref{UnfoldedCube}).  Thus $\beta$ is not induced by a homeomorphism taking $\Gamma$ to itself.  It follows that $\mathrm{TSG}(\Gamma) =\mathrm{TSG}_+(\Gamma) =A_4 \times \mathbb{Z}_2$.  \end{proof}

\begin{proposition}\label{Cor:D4Z2}$ D_4 \times \mathbb{Z}_2$, $ D_3\times \mathbb{Z}_2 $, and all of their subgroups are positively realizable and hence realizable for $P(4,1)$.\end{proposition}

\begin{proof} In the proof of Theorem~\ref{k2=1}, we saw that $D_n \rtimes \mathbb{Z}_2=\langle \rho, \delta, \alpha\rangle$, where $\rho$, $\delta$, and $\alpha$ are defined there.  Furthermore, $\langle \rho, \delta\rangle=D_n$, $\langle \alpha\rangle=\mathbb{Z}_2$, and $\alpha$ commutes with $\delta$.  When $k=1$, $\alpha$ also commutes with $\rho$.  Thus for $P(4,1)$, we have $\langle \rho, \delta, \alpha\rangle=D_4 \times \mathbb{Z}_2$.  It follows that $ D_4 \times \mathbb{Z}_2$ and all its subgroups are positively realizable for $P(4,1)$.

To show that $ D_3\times \mathbb{Z}_2 $ is positively realizable, we let $\Gamma$ be the skeleton of a cube with identical trefoil knots on the six edges containing $u_1$ or $v_3$.  A rotation by $\frac{2\pi}{3}$ around an axis through $u_1$ and $v_3$ induces the automorphism $\alpha=(u_2v_1u_4)(v_2v_4u_3)$, a rotation by $\pi$ around an axis containing $\overline{u_1v_1}$ and $\overline{u_3v_3}$ induces $\beta=(u_4u_2)(v_4v_2)$, and a rotation by $\pi$ around an axis through the middle of the edges $\overline{u_4v_4}$ and $\overline{u_2v_2}$ induces the automorphism $\delta=(u_4v_4)(u_2v_2)(u_1v_3)(v_1u_3)$.  Now $\langle \alpha,\beta, \delta\rangle=D_3\times \mathbb{Z}_2$, which is a maximal subgroup of $S_4 \times \mathbb{Z}_2$.

No homeomorphism can induce $(u_1u_2u_3u_4)(v_1v_2v_3v_4)$, since such a homeomorphism would send knotted edges to unknotted edges.  Thus we have $\mathrm{TSG}_+(\Gamma) =D_3\times \mathbb{Z}_2$.  Furthermore, the edge $\overline{u_4v_4}$ is not pointwise fixed by any non-trivial element of $\mathrm{TSG}_+(\Gamma)$, and hence  by Theorem \ref{SubgroupTheorem} every subgroup of $D_3\times \mathbb{Z}_2$ is positively realizable.  \end{proof}

In summary, we have proved the following.

\begin{theorem}
$\mathrm{Aut}(P(4,1))= S_4 \times \mathbb{Z}_2$ and all of its subgroups are positively realizable and hence realizable for $P(4,1)$.
\end{theorem}

\section{The exceptional case $P(8,3)$ }\label{P(8,3)}

 It follows from \cite{FRU1} that $\mathrm{Aut}(P(8,3)) = GL(2,3) \rtimes \mathbb{Z}_2= \langle \mu, \beta, \gamma \rangle $, where
\begin{align*}
\mu &= (u_1u_7v_8)(u_2v_7 v_5)(u_3v_4u_5)(u_6v_3v_1) \\
\beta &= (u_1u_7)(u_2u_6)(u_3u_5)(v_1v_7)(v_2v_6)(v_3v_5) \\
\gamma &=(u_1u_2u_3u_4u_5u_6u_7u_8)(v_1v_2v_3v_4v_5v_6v_7v_8).
\end{align*}
Throughout this section, we will use the embeddings of $P(8,3)$ illustrated in Figure \ref{Fig:AutP83}.  The bottom embedding is the same as the one on the right except that the edge $\overline{u_8v_8}$ passes through the point at $\infty$.  The steps in Figure \ref{Fig:AutP83} show that the embeddings are all isotopic.  Thus we abuse notation and refer to all of them as $\Gamma$.

\begin{figure}[h!]
\begin{center}\centering\includegraphics[scale=.11]{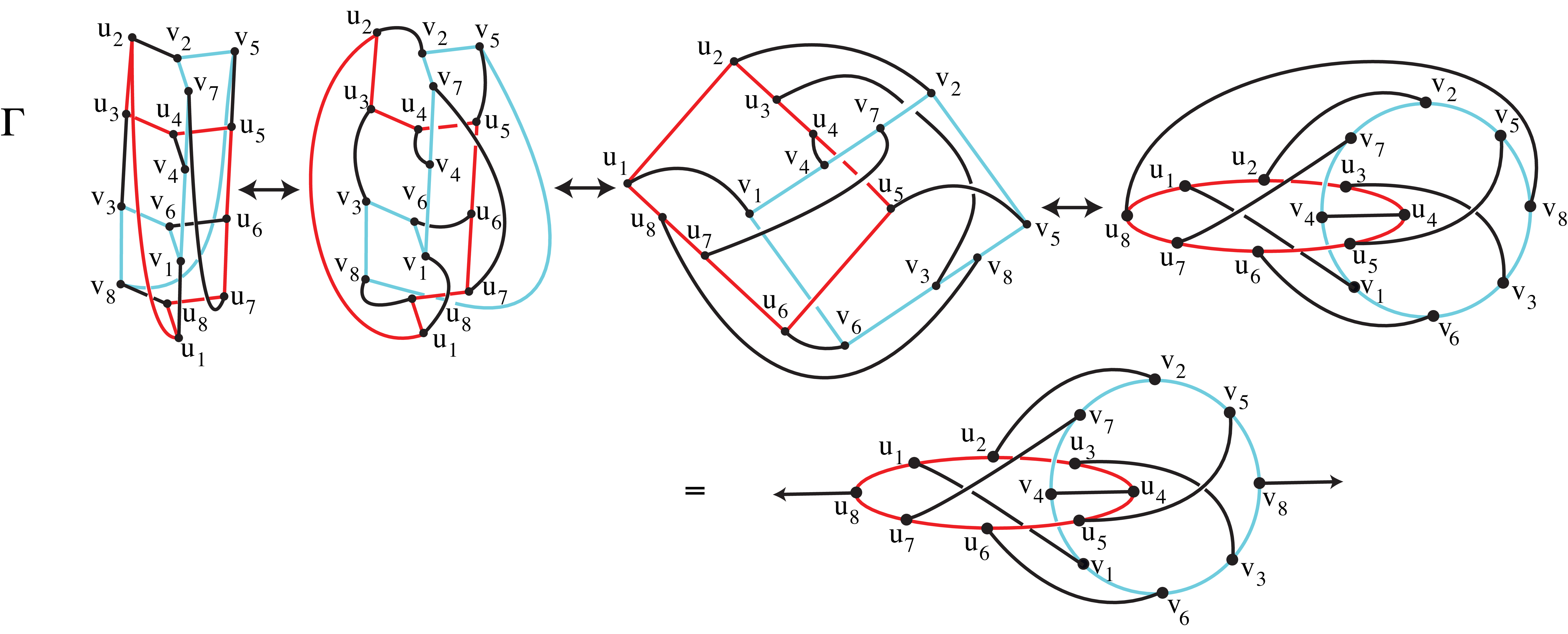}
\caption{ An embedding $\Gamma$ of $P(8,3)$ with $\mathrm{TSG}_+(\Gamma)= \mathrm{Aut}(P(8,3))$.  }
\label{Fig:AutP83}
\end{center}
\end{figure}

  Observe that $\mu$ is induced on the leftmost embedding by a rotation by $2\pi/3$ around the axis containing $u_8,v_6, u_4,v_2$. We see that $\gamma$ is induced on the right embedding by rotating the $u$-cycle by $\pi/4$ while rotating the $v$-cycle by $3\pi/4$.  Finally, $\beta$ is induced on the bottom embedding by a rotation by $\pi$ around the axis through $\infty$ containing the edges $\overline{u_4v_4}$ and $\overline{u_8v_8}$.  Since $\mathrm{Aut}(P(8,3)) = \langle \mu, \beta, \gamma \rangle$, we now have $\mathrm{TSG}_+(\Gamma)= \mathrm{Aut}(P(8,3))$.   Hence we have proved the following.

\begin{proposition} $GL(2,3) \rtimes \mathbb{Z}_2= \mathrm{Aut}(P(8,3))$ is positively realizable and hence realizable for $P(8,3)$.
\end{proposition}

All of the proper subgroups of $GL(2,3) \rtimes \mathbb{Z}_2$ are contained in the maximal groups $D_{12}$, $D_8 \rtimes \mathbb{Z}_2$, $GL(2,3)$, or $SL(2,3) \rtimes \mathbb{Z}_2$, which are addressed individually below.

\begin{proposition}
$D_{12}$ and all of its subgroups are positively realizable and hence realizable for $P(8,3)$.
\end{proposition}

\begin{proof} Let $\beta= (u_1u_7)(u_2u_6)(u_3u_5)(v_1v_7)(v_2v_6)(v_3v_5)$   
\[ \beta_1 = (u_1u_2u_3v_3v_8v_5u_5u_6u_7v_7v_4v_1)(u_4v_6u_8v_2). \]

Then $D_{12} = \langle \beta_1,\beta \rangle$ leaves the set of vertices $U=\{u_4,v_6,u_8,v_2 \}$ invariant, but $\mathrm{Aut}(P(8,3))$ does not leave $U$ invariant.  We create an embedding $\Gamma_{12}$ by adding identical $3_1$ knots to the edges of $\Gamma$ that include a vertex in $U$.  Then $\langle \beta_1 ,\beta \rangle\leq \mathrm{TSG}_+(\Gamma_{12})$.  Since $D_{12}$ is a maximal subgroup of $GL(2,3) \rtimes \mathbb{Z}_2$, $\mathrm{TSG}_+( \Gamma_{12}) =D_{12}$.    

Suppose that the edge $\overline{v_5u_5}$ of $\Gamma_{12}$ is pointwise fixed by an orientation preserving homeomorphism $h$ of $(S^3,\Gamma_{12})$.  Since $\overline{v_5v_2}$ is knotted and $\overline{v_5v_8}$ is not, both of these edges must also be pointwise fixed.  Now by Theorem~\ref{Rigid}, there is an embedding $\Gamma_{12}'$ of $P(8,3)$ and a finite order orientation preserving homeomorphism $f$ of $(S^3,\Gamma_{12}')$ inducing the same automorphism of $P(8,3)$ as $h$.  But by Smith Theory~\cite{SMI}, a triad cannot be pointwise fixed by $f$ unless $f$ is the identity.  Thus $\overline{v_5u_5}$ is not pointwise fixed by any non-trivial element of $\mathrm{TSG}_+( \Gamma_{12})$.  Applying Theorem \ref{SubgroupTheorem}, we conclude that all of the subgroups of $D_{12}$ are positively realizable for $P(8,3)$.  \end{proof}

The result below follows from Theorem~\ref{k2=1}.

\begin{proposition}
$D_8\rtimes \mathbb{Z}_2$ and all of its subgroups are positively realizable and hence realizable for $P(8,3)$.
\end{proposition}

\begin{proposition}
$GL(2,3)$ is positively realizable and hence realizable for $P(8,3)$.
\end{proposition}

\begin{proof} Let $A =\{ u_i, v_j \ | \ i \text { even, } j \text{ odd}\}$
and $B = \{ u_i, v_j \ | \ i \text { odd, } j \text{ even}\}$.
Starting with the embedding $\Gamma$ from Figure~\ref{Fig:AutP83}, we add the non-invertible knot $8_{17}$ to each edge between a vertex in $A$ and a vertex in $B$ oriented from $A$ to $B$.  This gives us an embedding $\Gamma_{GL}$.  Now $GL(2,3)=\langle \delta_1,\delta_2\rangle$ for
\begin{align*}
 \delta_1 &= (u_2v_1u_8)(u_3v_6v_8)(u_4u_6v_5)(u_7v_2v_4)  \\
\delta_2 &= (u_1u_3v_4v_2u_5u_7v_8v_6)(u_2u_4v_7v_5u_6u_8v_3v_1).
\end{align*}

Observe that $A$ and $B$ are each invariant under $\langle \delta_1,\delta_2\rangle$. Thus $\langle \delta_1,\delta_2\rangle \leq\mathrm{TSG}_+(\Gamma_{GL})$.  On the other hand, $\gamma =(u_1u_2u_3u_4u_5u_6u_7u_8)(v_1v_2v_3v_4v_5v_6v_7v_8)$ does not leave $A$ and $B$ invariant. Thus $\gamma$ cannot be induced by a homeomorphism of $(S^3, \Gamma_{GL})$, and hence $\gamma\not \in \mathrm{TSG}_+(\Gamma_{GL})$.  Since $\langle \delta_1,\delta_2 \rangle=GL(2,3)$ is a maximal subgroup of $\mathrm{Aut}(P(8,3))=GL(2,3) \rtimes \mathbb{Z}_2$, we have $\mathrm{TSG}_+(\Gamma_{GL}) = GL(2,3)$. \end{proof}

\begin{proposition}\label{SL2}
$SL(2,3) \rtimes \mathbb{Z}_2$ and all of its subgroups are positively realizable and hence realizable for $P(8,3)$.
\end{proposition}

\begin{proof} We start with the embedding $\Gamma$ of $P(8,3)$ illustrated in Figure~\ref{Fig:AutP83}, which has $\mathrm{TSG}_+(\Gamma)=\mathrm{Aut}(P(8,3))=GL(2,3) \rtimes \mathbb{Z}_2$.  Now by Theorem~\ref{GroupRigid}, there is an embedding $\Gamma'$ such that
$\mathrm{Aut}(P(8,3))\leq \mathrm{TSG}_+(\Gamma')\leq\mathrm{Aut}(P(8,3))$ which is induced by an isomorphic group $G$ of orientation preserving isometries of $(S^3,\Gamma')$.

Now we modify $\Gamma$ in a neighborhood of the vertex $u_1$ so that it looks like the neighborhood $N(u_1)$ illustrated on the left in Figure~\ref{Fig:orient}.  We say that an edge $e_1$ is {\it knotted around} an edge $e_2$, if there is a ball whose boundary intersects the graph in four points giving us the tangle on the right in Figure~\ref{Fig:orient}, where the $3_1$ knot in one string can be replaced by any non-trivial knot which is linked with the other string.  Thus after changing $\Gamma$ in the neighborhood $N(u_1)$, we see that $\overline{u_1u_2}$ is knotted around $\overline{u_1u_8}$ which is knotted around $\overline{u_1v_1}$ which is knotted around $\overline{u_1u_2}$.  Since none of the reverse knottings hold, this gives an \emph{order to the edges} around $u_1$.    

\begin{figure}[h!]
\begin{center}\centering\includegraphics[scale=.2]{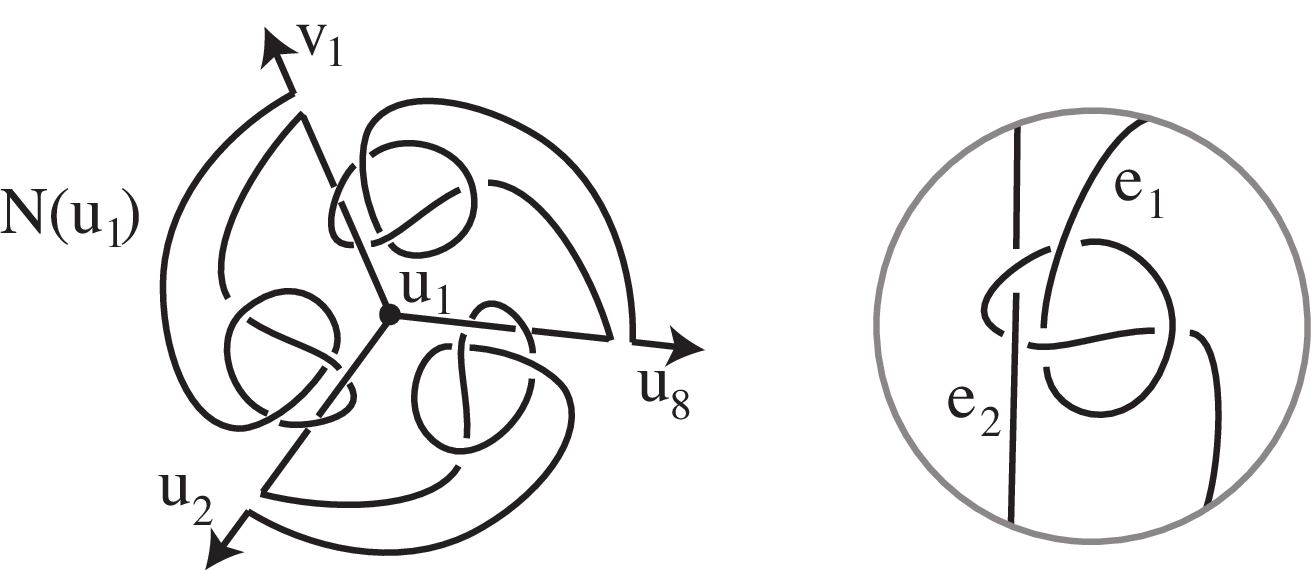}
\caption{ We embed the edges near $u_1$ as on the left.  We say $e_1$ is \emph{knotted around} $e_2$ if there is a tangle as on the right.  }
\label{Fig:orient}
\end{center}
\end{figure}

Observe that $SL(2,3) \rtimes \mathbb{Z}_2=\langle \rho_1,\rho_2,\rho_3 \rangle $ where 
\begin{align*}
\rho_1 &= (u_2v_1u_8)(u_3v_6v_8)(u_4u_6v_5)(u_7v_2v_4) \\
\rho_2 &= (u_1v_3u_5v_7)(u_2v_8u_6v_4)(u_3v_5u_7v_1)(u_4v_2u_8v_6)\\
\rho_3  &= (u_1u_7v_8)(u_2v_7v_5)(u_3v_4u_5)(u_6v_3v_1).
\end{align*}

 Since $\langle \rho_1,\rho_2,\rho_3 \rangle $ acts transitively on the vertices, we can apply the isometry group $G $ to the neighborhood $N(u_1)$ to modify a neighborhood around every vertex of $\Gamma'$.  To check that this is well-defined, we used Sage to determine that the only non-trivial automorphisms in $\langle \rho_1,\rho_2,\rho_3 \rangle $ which fix a vertex are of order $3$.  Hence no edge is pointwise fixed by any non-trivial element of $\langle \rho_1,\rho_2,\rho_3 \rangle $.  It follows that no automorphism in $\langle \rho_1,\rho_2,\rho_3 \rangle $ changes the order of the edges sharing a vertex.   Thus our modification of the neighborhoods around vertices is well-defined, and hence we have a well-defined embedding $\Gamma_{SL2}$ of $P(8,3)$ such that the edges around every vertex have an order.  Hence $SL(2,3) \rtimes \mathbb{Z}_2=\langle \rho_1,\rho_2,\rho_3 \rangle\leq\mathrm{TSG}_+(\Gamma_{SL2})$.

 Now $\beta= (u_1u_7)(u_2u_6)(u_3u_5)(v_1v_7)(v_2v_6)(v_3v_5)\in \mathrm{Aut}(P(8,3)=GL(2,3) \rtimes \mathbb{Z}_2$ does not preserve the order of edges around $u_4$ since $\beta$ fixes $\overline{u_4u_5}$ and interchanges $u_3$ and $u_5$.  Thus $\mathrm{TSG}_+(\Gamma_{SL2})\not=\mathrm{Aut}(P(8,3))$.  Since $SL(2,3) \rtimes \mathbb{Z}_2$ is a maximal subgroup of $\mathrm{Aut}(P(8,3))$, it follows that $\langle \rho_1,\rho_2,\rho_3 \rangle=\mathrm{TSG}_+(\Gamma_{SL2})$.  Now since no edge of $P(8,3)$ is pointwise fixed by any non-trivial element of $\langle \rho_1,\rho_2,\rho_3 \rangle$, we can apply Theorem~\ref{SubgroupTheorem} to conclude that all of the subgroups of $SL(2,3) \rtimes \mathbb{Z}_2$ are also positively realizable. \end{proof}

In summary, we have proved the following.

\begin{theorem}
$\mathrm{Aut}(P(8,3))= GL(2,3) \rtimes \mathbb{Z}_2$ and all of its subgroups are positively realizable and hence realizable for $P(8,3)$.
\end{theorem}

\section{The exceptional case $P(10,2)$}\label{P(10,2)}

Frucht proved that $\mathrm{Aut}(P(10,2)) = A_5 \times \mathbb{Z}_2$ \cite{FRU}.  The embedding $\Gamma$ of $P(10,2)$ as the 1-skeleton of a regular dodecahedron (see the left image of \ref{Fig:dodec}) will be the basis for all of the embeddings in this section.  The group $G=A_5$ of rotations of a solid dodecahedron consists of six order $5$ rotations about an axis through the centers of opposite faces, ten order $3$ rotations about an axis through opposite vertices, and 15 order $2$ rotations about an axis through midpoints of opposite edges.  

\begin{figure}[h!]
\begin{center}\centering\includegraphics[scale=.24]{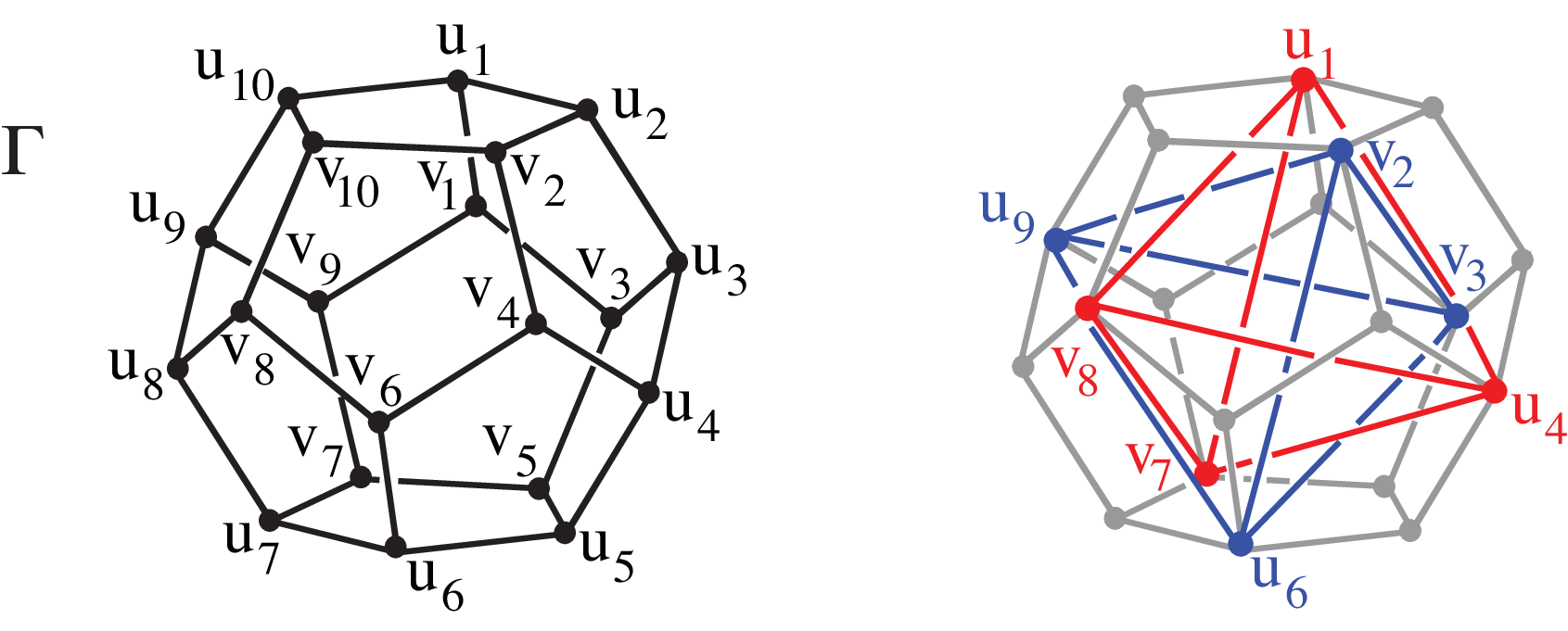}
\caption{On the left, $\Gamma$ is an embedding of $P(10,2)$ with $\mathrm{TSG}_+(\Gamma) = A_5 \times \mathbb{Z}_2$.  On the right, red and blue inscribed tetrahedra which will be used in the proof of Proposition~\ref{A4xZ2realizable}. }
\label{Fig:dodec}
\end{center}
\end{figure}

The automorphism $\alpha=(v_2v_8)(u_2u_8)(v_{6}v_4)(u_{6}u_4)(v_3v_7)(u_3u_7)(v_1v_9)(u_1u_9)$ is not induced by a rotation of the solid dodecahedron, and hence is not in $G$.  However, it is induced by a rotation $h$ of $S^3$ around the equator of $\Gamma$ containing $\overline{u_{10}v_{10}}$ and $\overline{u_5v_5}$ which interchanges the inside and outside of the dodecahedron.  Thus $H= \langle G, h\rangle=A_5 \times \mathbb{Z}_2$ is a group of rotations of $(S^3, \Gamma)$ which induces $\mathrm{Aut}(P(10,2)) = A_5 \times \mathbb{Z}_2$.  Hence we have proved the following.

\begin{proposition}\label{A5xZ2posrealizable}
	$\Aut(P(10,2)) = A_5 \times \mathbb{Z}_2$ is positively realizable and hence realizable for $P(10,2)$.
\end{proposition}

Note that by Theorem~\ref{Thm:Gen_Dn}, we have the following.

\begin{proposition}\label{subgroupsofA5xz2}
	$D_{10}$ and all of its subgroups are positively realizable and hence realizable for $P(10,2)$.
\end{proposition}

\begin{proposition}\label{A4xZ2realizable}
	$A_4 \times \mathbb{Z}_2$ and all of its subgroups are positively realizable and hence realizable for $P(10,2)$.
\end{proposition}
\begin{proof} Let $T_1=\{u_{1},u_{4},v_{7},v_{8}\}$ and $T_2=\{u_6, u_9, v_2,v_3\}$ be the red and blue sets of vertices respectively illustrated on the right in Figure~\ref{Fig:dodec}.  The vertices in each $T_i$ are equidistant from each other and form the corners of a solid tetrahedron inscribed in the solid dodecahedron.  Let $G'$ be the subgroup of the group $G$ of rotations of the solid dodecahedron which takes each of these solid tetrahedron to itself inducing the group of rotations $A_4$ on both tetrahedra.   The rotation $h$ of $S^3$ around the equator of $\Gamma$ containing $\overline{u_{10}v_{10}}$ and $\overline{u_5v_5}$ interchanges the vertices in $T_1$ and $T_2$.  Thus $H'=\langle G',h\rangle$ is a group of rotations of $(S^3, \Gamma)$ which induces $A_4 \times \mathbb{Z}_2$ on $\Gamma$. 

Now we modify the embedding $\Gamma$ by adding the $4_1$ knot to each edge which has a vertex in $T_1\cup T_2$ to get an embedding $\Gamma_{A42}$ which is invariant under $H'$.  Thus $A_4 \times \mathbb{Z}_2\leq \mathrm{TSG}_+(\Gamma_{A42})$, and every homeomorphism of $(S^3, \Gamma_{A42})$ leaves $T_1\cup T_2$ setwise invariant.  Now observe that no automorphism of order $5$ of $P(10,2)$ leaves $T_1\cup T_2$ setwise invariant.  Thus no order $5$ automorphism is in $\mathrm{TSG}_+(\Gamma_{A42})$.  Since $A_4 \times \mathbb{Z}_2$ is a maximal subgroup of $A_5 \times \mathbb{Z}_2$, we have $A_4 \times \mathbb{Z}_2= \mathrm{TSG}_+(\Gamma_{A42})$ induced by $H'$.  

Finally, suppose that the edge $\overline{v_2v_4}$ is pointwise fixed by a non-trivial element of $\mathrm{TSG}_+(\Gamma_{A42})$.   Then there is an element $f\in H'$ which pointwise fixes $\overline{v_2v_4}$.  Since $\overline{v_4u_4}$ has a knot in it and $\overline{v_4v_6}$ does not, both of these edges must also be pointwise fixed by $f$.  Since all the elements of $H'$ are rotations, this means that $f$ must be trivial. Thus $\overline{v_2v_4}$ is not pointwise fixed by any non-trivial element of $\mathrm{TSG}_+(\Gamma_{A42})$. Hence we can apply Theorem~\ref{SubgroupTheorem} to conclude that all of the subgroups of $A_4 \times \mathbb{Z}_2$ are positively realizable for $P(10,2)$.\end{proof}

\begin{proposition}\label{propersubgroupsA5realizable}  $A_5$ and its subgroups are positively realizable and hence realizable for $P(10,2)$.
\end{proposition}

\begin{proof} Starting with $\Gamma$ we embed the edges around a neighborhood of $u_1$ as in $N(u_1)$ in Figure~\ref{Fig:orient}.  Now as in the proof of Proposition~\ref{SL2}, we apply $G$ to $N(u_1)$ to modify a neighborhood around every vertex.  Since $G$ is the group of rotations of the solid dodecahedron, no non-trivial element of $G$ changes the order of edges around any vertex.    Hence this gives us a well-defined embedding $\Gamma_{A5}$ of $P(10,2)$ which is invariant under $G$.  Observe that $\alpha$  changes the order of edges around $u_9$ and hence $\alpha\not\in \mathrm{TSG}_+(\Gamma_{A5})$.   Because $A_5$ is a maximal subgroup of $\mathbb{Z}_5 \times \mathbb{Z}_2$, it follows that $\mathrm{TSG}_+(\Gamma_{A5})=A_5$.  

Finally, since no non-trivial element of $G$ pointwise fixes any edge of $\mathrm{TSG}_+(\Gamma_{A5})$, we can apply Theorem  \ref{SubgroupTheorem} to conclude that every subgroup of $A_5$ is positively realizable for $P(10,2)$.
\end{proof}

\begin{proposition}\label{D6prop}
	$D_6$ is positively realizable and hence realizable for $P(10,2)$.
\end{proposition}
\begin{proof}

We again start with the embedding $\Gamma$ of $P(10,2)$ from Figure~\ref{Fig:dodec}.     Let $h_1$ be a rotation of $(S^3,\Gamma)$ of order $2$ around the equator of $\Gamma$ containing the edges $\overline{u_{10}v_{10}}$ and $\overline{u_5v_5}$.  Then $h_1$ induces the automorphism $$\alpha_1=(u_1 u_9)(u_2u_8)(u_3u_7)(u_4u_6)(v_1 v_9)(v_2v_8)(v_3v_7)(v_4v_6).$$ Let $h_2$ be a rotation of $(S^3,\Gamma)$ of order $2$ around an axis that passes through the midpoints of the edges $\overline{u_7u_8}$ and $\overline{u_2u_3}$.  Then $h_2$ induces the automorphism $$\alpha_2=(u_2u_3)(u_6u_9)(u_7u_8)(u_5u_{10})(u_4u_1)(v_2v_3)(v_6v_9)(v_7v_8)(v_5v_{10})(v_4v_1).$$  It follows that $h=h_2h_1$ is an isometry of $(S^3,\Gamma)$ inducing the automorphism $$\alpha'=\alpha_2\alpha_1=(u_1u_6)(u_2u_7)(u_3u_8)(u_4u_9)(u_5u_{10})(v_1v_6)(v_2v_7)(v_3v_8)(v_4v_9)(v_5v_{10}).$$  Let $f$ be an order $3$ rotation of $(S^3, \Gamma)$ around an axis that passes through  vertices $u_1$ and $u_6$.  Then $f$ induces the automorphism $$\beta=(u_2u_{10} v_1)(u_3v_{10}v_9)(u_4 v_8 v_7)(u_5v_6u_7)(u_8v_5v_4)(u_9v_3 v_2).$$  Finally, let $g$ be the order $2$ rotation of $(S^3,\Gamma)$ around an axis that passes through the midpoints of edges $\overline{u_3u_4}$ and $\overline{u_8u_9}$.  Then $g$ induces the automorphism $$\gamma=(u_1u_6)(u_2 u_5)(u_3 u_4)(u_7 u_{10})(u_8 u_9)(v_1 v_6)(v_2 v_5)(v_3 v_4)(v_7 v_{10})(v_8 v_9).$$

Note that the pair $\{u_1,u_6\}$ is setwise fixed by $\alpha'$, $\beta$, and $\gamma$.  Now we add $4_1$ knots to the six edges containing $u_1$ or $u_6$ to obtain an embedding $\Gamma_{D6}$ such that the isometries $h$, $g$, and $f$ leave $\Gamma_{D6}$ setwise invariant.  Then $\langle h,g, f\rangle$ induces $\langle \alpha',\beta, \gamma\rangle$ on $\Gamma_{D6}$.  Since no non-trivial finite order orientation preserving isometry of $S^3$ can pointwise fix $\Gamma_{D6}$,  the isometry group $\langle h,g, f\rangle$ is isomorphic to the automorphism group $\langle \alpha',\beta, \gamma\rangle$.  Furthermore, because of the $4_1$ knots on the edges of the triads centered at $u_1$ and $u_6$, every element of $\langle h,g, f\rangle$ leaves this pair of triads  setwise invariant.  Since no non-trivial finite order orientation preserving isometry  can pointwise fix a triad, $\langle h,g, f\rangle$ induces an isomorphic action on this  pair of triads.  

However, the action of $\langle \alpha',\beta, \gamma\rangle$ on the two triads is $D_6$ because $\alpha'$ interchanges the two triads, $\beta$ rotates both triads, and $\alpha'\gamma$ flips over each triad fixing the edges $\overline{u_6v_6}$ and $\overline{u_1v_1}$. Thus $D_6\leq \mathrm{TSG}_+(\Gamma_{D6})$.  On the other hand, because of the $4_1$ knots, not every element of $\Aut(P(10,2))$ can be induced on the embedding $\Gamma_{D6}$.  It follows that $\mathrm{TSG}_+(\Gamma_{D6})$ is a proper subgroup of $\Aut(P(10,2)) = A_5 \times \mathbb{Z}_2$.  Since $D_6$ a maximal in $A_5 \times \mathbb{Z}_2$, this means that $D_6= \mathrm{TSG}_+(\Gamma_{D6})$.  \end{proof}


In summary, we have proven  the following.

\begin{theorem}
$\Aut(P(10,2)) = A_5 \times \mathbb{Z}_2$ and all of its subgroups are positively realizable and hence realizable for $P(10,2)$. 
	\end{theorem}

\section{The exceptional case $P(10,3)$ }\label{P(10,3)}


The group $\mathrm{Aut}(P(10,3))=S_5 \times \mathbb{Z}_2= \langle \alpha,\beta\rangle$ where
\begin{align*}
\alpha &= (u_1u_2u_3u_4u_5u_6u_7u_8u_9u_{10})  (v_1 v_2 v_3 v_4 v_5 v_6 v_7 v_8 v_9 v_{10}) \\
\beta &= (u_1 v_4)(u_2 u_4)(u_5 v_2)(u_6 v_9)(u_7 u_9)(u_{10} v_7).
\end{align*}

The isomorphism classes of proper nontrivial subgroups of $S_5 \times \mathbb{Z}_2$ are $A_5 \times \mathbb{Z}_2, S_5, A_5, S_4 \times \mathbb{Z}_2, (\mathbb{Z}_5 \rtimes \mathbb{Z}_2) \times \mathbb{Z}_2, S_3 \times \mathbb{Z}_2^2, A_4 \times \mathbb{Z}_2, S_4, D_{10}, \mathbb{Z}_5 \rtimes \mathbb{Z}_2, D_4 \times \mathbb{Z}_2, \mathbb{Z}_6 \times \mathbb{Z}_2, D_6, A_4, \mathbb{Z}_{10}, D_5, \mathbb{Z}_2^3, D_4, \mathbb{Z}_4 \times \mathbb{Z}_2, \mathbb{Z}_6, S_3, \mathbb{Z}_5, \mathbb{Z}_2^2, \mathbb{Z}_4, \mathbb{Z}_3$ and  $\mathbb{Z}_2$.  We will see that all of these subgroups are realizable, but not all are positively realizable for $P(10,3)$. We begin by determining which subgroups are not positively realizable.

\begin{proposition}
$\mathbb{Z}_{5} \rtimes \mathbb{Z}_4$ , $\mathbb{Z}_{10} \rtimes \mathbb{Z}_4$ and $S_5 \times \mathbb{Z}_2$ are not positively realizable for $P(10,3)$.
\end{proposition}
\begin{proof}  Suppose that $\mathbb{Z}_{5} \rtimes \mathbb{Z}_4\leq \mathrm{TSG}_+(\Lambda)$ for some embedding $\Lambda$ of $P(10,3)$ in $S^3$.  Then by Theorem~\ref{GroupRigid}, for some embedding $\Gamma$ of $P(10,3)$ in $S^3$, the group $\mathbb{Z}_{5} \rtimes \mathbb{Z}_4$ is induced on $\Gamma$ by an isomorphic group of orientation preserving isometries $G$.

Using Sage, we determined the elements of $\mathrm{Aut}(P(10,3))$.  In particular, every involution in $\mathbb{Z}_{5} \rtimes \mathbb{Z}_4$ fixes either two or four vertices.  Thus for every involution $g\in G$ we have $\mathrm{fix}(g)=S^1$.  Since no element of $G$ of order $4$ or $5$ fixes any vertices, no other element of $G$ has the same fixed point set as an involution.  Hence $G$ satisfies the hypothesis of Theorem~\ref{involution}, but $\mathbb{Z}_{5} \rtimes \mathbb{Z}_4$ is not one of the groups in the conclusion of the theorem.  Thus $\mathbb{Z}_{5} \rtimes \mathbb{Z}_4$ is not contained in $\mathrm{TSG}_+(\Lambda)$ for any embedding $\Lambda$ of $P(10,3)$ in $S^3$.  Now since $\mathbb{Z}_{10} \rtimes \mathbb{Z}_4$ and $S_5 \times \mathbb{Z}_2$ each contain $\mathbb{Z}_{5} \rtimes \mathbb{Z}_4$ as a subgroup, they also cannot be positively realizable.
\end{proof}

\begin{lemma}\label{neg246}  The following automorphisms are not positively realizable for $P(10,3)$.
\begin{itemize}
\item An order $2$ automorphism with only six $2$-cycles.
\item  An order $4$ automorphism with $2$-cycles which contain adjacent vertices.
\item An order $6$ automorphism with $3$-cycles.
\end{itemize}
\end{lemma}

\begin{proof} Suppose that an automorphism of $P(10,3)$ is positively realizable.  Then by Theorem~\ref{Rigid}, that automorphism is induced on some embedding by a finite order orientation preserving homeomorphism of $S^3$. By Smith Theory \cite{SMI}, such a homeomorphism either pointwise fixes an $S^1$ or is fixed point free.

All order $2$ automorphisms that contain only six $2$-cycles are conjugate to $$(u_1 v_4)(u_2 u_4)(u_5 v_2)(u_6 v_9)(u_7 u_9)(u_{10} v_7)$$ which pointwise fixes the edges $\overline{v_3u_3}$, $\overline{v_3v_{10}}$, $\overline{v_3v_6}$, $\overline{v_8u_8}$, $\overline{v_8v_5}$, $\overline{v_8,v_{1}}$.  Since these six edges form two triads, they cannot be contained in an $S^1$.  Thus no such order $2$ automorphism can be positively realizable.

An order $4$ automorphism that contains $2$-cycles with adjacent vertices, flips over the edges between such vertices, pointwise fixing their midpoints.  Suppose that such an automorphism is induced on some embedding by a finite order orientation preserving homeomorphism $h$ of $S^3$.  Then $\mathrm{fix}(h)$ is an $S^1$ and so is $\mathrm{fix}(h^2)$.  But $\mathrm{fix}(h)\subseteq \mathrm{fix}(h^2)$ and hence these sets are equal.  This is impossible since $h^2$ pointwise fixes edges which are flipped over by $h$.

All order $6$ automorphisms that contain $3$-cycles are conjugate to the automorphism $\rho=(u_1,v_2,u_3)(u_4 u_{10} v_5 v_3 v_1 v_9)(u_5 v_{10} v_8 v_6 v_4 u_9)(u_6 v_7 u_8)$.  If $\rho$ were positively realizable, then $\rho^3$ would be as well.  But $\rho^3$ is an order $2$ automorphism that contains only six $2$-cycles, which we saw is not positively realizable.\end{proof}

\begin{proposition}\label{consequences246} $\mathbb{Z}_6\times \mathbb{Z}_2$, $D_6\times \mathbb{Z}_2$, $S_4$, $S_5$, $S_4\times \mathbb{Z}_2$, $\mathbb{Z}_4 \times \mathbb{Z}_2$, and $D_4 \times \mathbb{Z}_2$ are not positively realizable for $P(10,3)$.
\end{proposition}

\begin{proof} Using Sage, we determined that there are two conjugacy classes of subgroups of $\mathrm{Aut}(P(10,3))$ that are isomorphic to $S_4$.  One contains an order $2$ automorphism with only six $2$-cycles while the other contains an order $4$ automorphism with a $2$-cycle with adjacent vertices.  Thus by Lemma~\ref{neg246}, neither is positively realizable.  Since $S_5$ and $S_4\times \mathbb{Z}_2$ both contain $S_4$, they too are not positively realizable.

We also determined with Sage that $ \mathrm{Aut}(P(10,3))$ has only one conjugacy class isomorphic to $\mathbb{Z}_6\times \mathbb{Z}_2$ and it contains an order $6$ automorphism that includes $3$-cycles.  Thus by Lemma~\ref{neg246}, it is not positively realizable.  Since $D_6\times \mathbb{Z}_2$ contains $\mathbb{Z}_6\times \mathbb{Z}_2$, it too is not positively realizable.

Finally, we determined that there is only one conjugacy class of subgroups isomorphic to $\mathbb{Z}_4 \times \mathbb{Z}_2$ and it contains an order $4$ element with $2$-cycles with adjacent vertices.  Thus by Lemma~\ref{neg246}, it is not positively realizable.  Since $D_4 \times \mathbb{Z}_2$ contains $\mathbb{Z}_4 \times \mathbb{Z}_2$, it too is not positively realizable.\end{proof}

Below we determine the positively realizable subgroups of $\mathrm{Aut}(P(10,3))$.  By Theorem \ref{Thm:Gen_Dn} we have Proposition~\ref{d10}.

\begin{proposition}\label{d10}
$D_{10}$ and all of its subgroups are positively realizable and hence realizable for $P(10,3)$.
\end{proposition}

\begin{proposition} \label{P103_pos_D4} $\mathbb{Z}_4$ and
$D_4$ are positively realizable and hence realizable for $P(10,3)$.
\end{proposition}

\begin{proof} Since $\mathbb{Z}_4\leq\mathbb{Z}_{10} \rtimes \mathbb{Z}_4=B(10,3)$, it follows from Theorem~\ref{order4} that $\mathbb{Z}_4$ is positively realizable.  Now we consider $D_4=\langle \nu_1,\nu_2 \rangle$, where 
\begin{align*}
\nu_1 &=(u_1 u_5)(u_2 v_5 v_1 u_4)(u_3 v_2 v_8 v_4)(u_6 u_{10})(u_7 v_{10} v_6 u_9)(u_8 v_7 v_3 v_9) \\
\nu_2 &= (u_1 u_6)(u_2 u_7)(u_3 v_7)(u_4 v_{10})(u_5 u_{10})(u_8 v_2)(u_9 v_5)(v_1 v_6)(v_3 v_4)(v_8 v_9).
\end{align*}

\smallskip

 Let $\Gamma_{D4}$ be the embedding of $P(10,3)$ in a standardly embedded solid torus $T$ in $S^3$, illustrated in Figure~\ref{D4}, where the edges which are on $\partial T$ are  green on the flat torus  and the edges of $\Gamma_{D4}$ which go into the interior of $T$ are red.  The edges $\overline{u_{10}u_1}$ and $\overline{u_5u_6}$ are diametrically opposed arcs on the core of the solid torus illustrated in the center.  The vertices in a given orbit of $\nu_1$ are the same color.  Two special meridians are colored purple and orange on the right.

\begin{figure}[h!] 
\begin{center}
\centering\includegraphics[scale=.28]{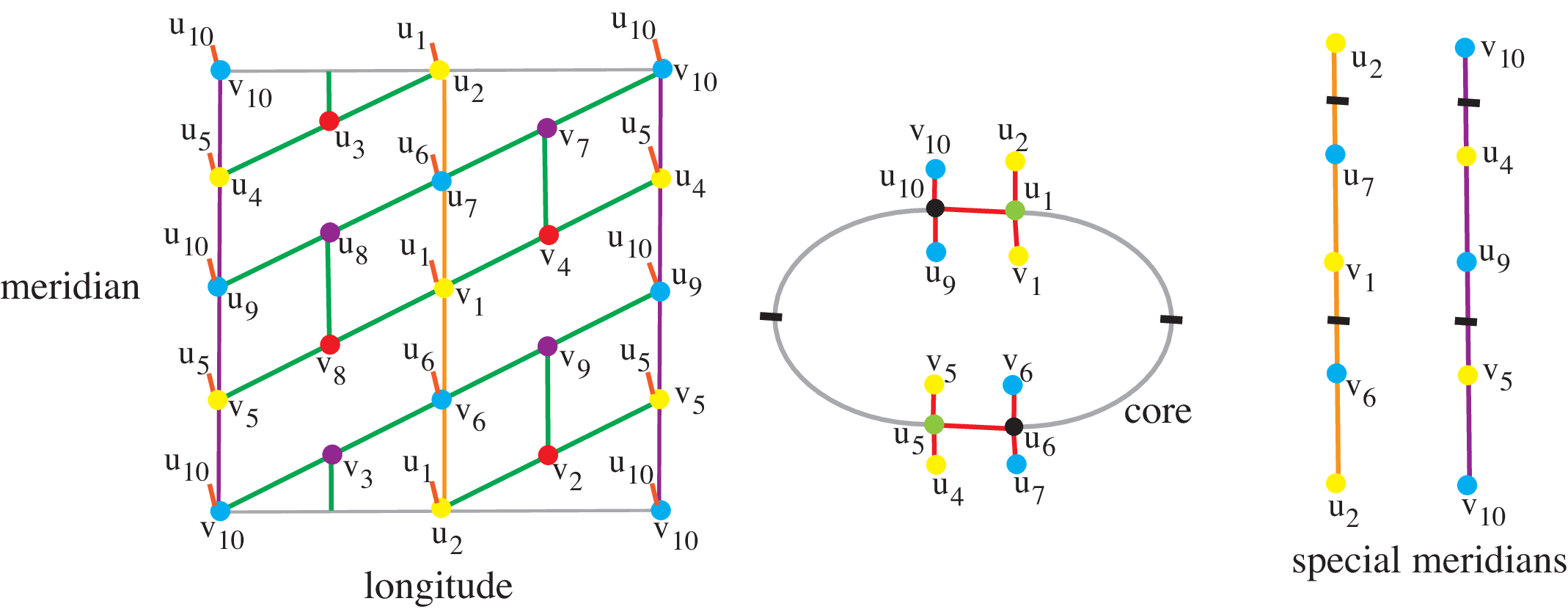}
\end{center}
\caption{On the left and center, is an embedding $\Gamma_{D_4}$ of $P(10,3)$ on the boundary and core of a standardly embedded solid torus such that $\mathrm{TSG}_+(\Gamma_{D_4}) = D_4$.  On the right, we display two special meridians.}
\label{D4}
\end{figure}

The automorphism $\nu_1$ is induced by the glide rotation $h_{\nu_1}$ which rotates $T$ meridionally by $\pi/2$ around its core and rotates $T$ longitudinally by $\pi$ around the core of a complementary solid torus.  The automorphism $\nu_2$ is induced by a rotation $h_{\nu_2}$ by $\pi$ around an axis which pierces $\partial T$ in two points on the orange meridian, two points on the purple meridian, and two points on the core of $T$, as indicated by black dashes in Figure~\ref{D4}.  Now $\langle h_{\nu_1}, h_{\nu_2}\rangle$ induces $D_4=\langle \nu_1,\nu_2 \rangle$ on $\Gamma_{D4}$.  Thus $D_4\leq \mathrm{TSG}_+(\Gamma_{D_4})$. If  $\mathrm{TSG}_+(\Gamma_{D_4}) \not= D_4$, then $\mathrm{TSG}_+(\Gamma_{D_4})$ would contain a group which has $D_4$ as a maximal subgroup.  However, the only subgroups of $\mathrm{Aut}(P(10,3))=S_5 \times \mathbb{Z}_2$ which contain $D_4$ as a maximal subgroup are $D_4 \times \mathbb{Z}_2$ and $S_4$, and we saw in the proof of Proposition~\ref{consequences246} that each of these groups contains an automorphism which is not positively realizable for $P(10,3)$.  It follows that $\mathrm{TSG}_+(\Gamma_{D_4}) = D_4$.  \end{proof}

\begin{proposition} \label{A5_edge_embed}
$A_5 \times \mathbb{Z}_2$  and $A_5$ are positively realizable and hence realizable for $P(10,3)$.
\end{proposition}

\begin{proof} Let $\Omega$ denote a regular $4$-simplex embedded in $S^3$.   The group $G$ of orientation preserving isometries of $(S^3,\Omega)$ is $A_5$.   We embed half of the vertices of $P(10,3)$ as midpoints of the edges of $\Omega$ (as illustrated on the left of Figure~\ref{simplex}) and the other half as center points of the faces of $\Omega$.  Then we connect vertices with straight edges to obtain an embedding $\Lambda$ of $P(10,3)$.  We illustrate $\Lambda$ in an unfolded picture of $\Omega$ in the center of Figure~\ref{simplex}. Since $\Lambda$ is symmetrically embedded in $\Omega$, it is setwise invariant under any isometry of $\Omega$, in particular under $G$.  

\begin{figure}[h!]
\begin{center}
\centering\includegraphics[scale=.25]{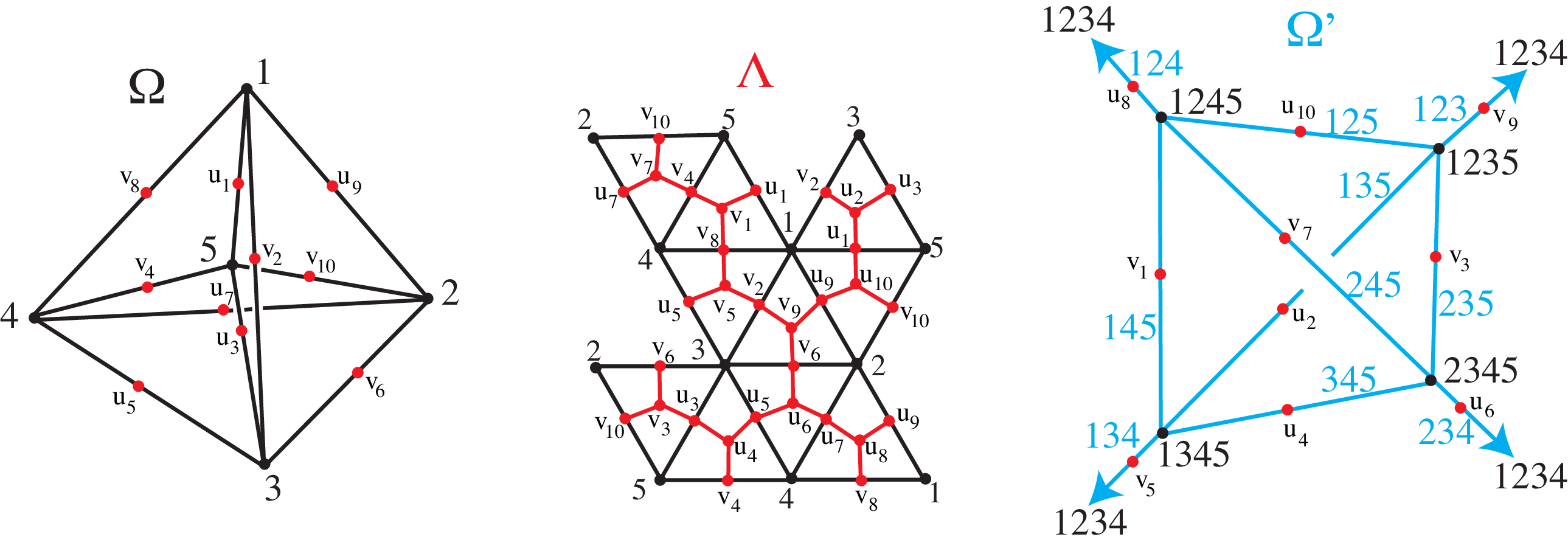}
\end{center}
\caption{The $4$-simplex on the left is $\Omega$.  The red graph in the center is the embedding $\Lambda$ of $P(10,3)$ in an unfolded picture of $\Omega$ such that $ \mathrm{TSG}_+(\Lambda)=A_5 \times \mathbb{Z}_2$.  On the right is the $4$-simplex $\Omega'$ which is dual to $\Omega$.}
\label{simplex}
\end{figure}

We define the dual $4$-simplex $\Omega'$ in $S^3$ as follows.  We embed a vertex of $\Omega'$ as the center point of each tetrahedron of $\Omega$ and we embed each edge of $\Omega'$ as a straight segment joining two vertices.  Then each edge of $\Omega'$ will pass through the center point of a face of $\Omega$.  Next we embed each face of $\Omega'$ as the triangle bounded by  three pairwise adjacent edges of $\Omega'$.  Then the midpoint of each edge of $\Omega$ will pass through the center point of a face of $\Omega'$.  Finally, we define a tetrahedron of $\Omega'$ as the solid bounded by four pairwise adjacent faces of $\Omega'$.  Thus each vertex of $\Omega$ will be the center point of a tetrahedron of $\Omega'$.  In Figure~\ref{simplex}, we illustrate  $\Omega'$ in blue on the right.  We label each vertex of $\Omega'$ by the tetrahedron of $\Omega$ which it is the center of, and we label each edge of $\Omega'$ by the face of $\Omega$ which it intersects.  Note that in the figure the vertex $1234$ of $\Omega'$ is at the point at $\infty$.  We see from the center image that half of the vertices of $\Lambda$ are on edges of $\Omega$ and the other half are on edges of $\Omega'$.  Also, all of the edges of $\Lambda$ go between $\Omega$ and $\Omega'$.

Now there is a glide rotation $h$ of $S^3$ which interchanges $\Omega$ and $\Omega'$, interchanging each vertex $i\in\Omega$ with the vertex $jklm\in\Omega'$ such that $i\not\in \{j,k,l,m\}$, and $h$ takes $\Lambda$ to itself inducing the automorphism $$(u_1u_6)(u_2u_7)(u_3u_8)(u_4u_9)(u_5u_{10})(v_1v_6)(v_2v_7)(v_3v_8)(v_4v_9)(v_5v_{10}).$$  Since $h$ commutes with every element of the group $G$ of isometries,  $H=\langle G, h\rangle=A_5 \times \mathbb{Z}_2$, and $H$ leaves $\Lambda$ setwise invariant.  Hence $A_5 \times \mathbb{Z}_2\leq \mathrm{TSG}_+(\Lambda)$.  Now since $A_5 \times \mathbb{Z}_2$ is a maximal subgroup of $S_5 \times \mathbb{Z}_2$ and $S_5 \times \mathbb{Z}_2$ is not positively realizable for $P(10,3)$, we have $A_5 \times \mathbb{Z}_2= \mathrm{TSG}_+(\Lambda)$.  Thus $A_5 \times \mathbb{Z}_2$ is positively realizable.

We obtain a new embedding $\Lambda'$ by adding the non-invertible knot $8_{17}$ to every edge of $\Lambda$ such that the orientation of the knots goes from the vertices on edges of $\Omega$ towards the vertices on edges of $\Omega'$.  Then $\Lambda'$ is setwise invariant under the group $G$ of orientation preserving isometries of $(S^3,\Omega)$, but no element of $\mathrm{TSG}_+(\Lambda)$ interchanges the sets of vertices on edges of $\Omega$ with those on edges of $\Omega'$.  It follows that $ \mathrm{TSG}_+(\Lambda')=A_5$.  Hence, $A_5 $ is positively realizable for $P(10,3)$. \end{proof}

\begin{proposition} \label{P103_D6_pos}
$D_6$ and all of its subgroups are positively realizable and hence realizable for $P(10,3)$.
\end{proposition}

\begin{proof}
$D_6=\langle \theta_1,\theta_2 \rangle$ where
\begin{align*}
\theta_1 &= (u_1 u_4 v_{10} u_6 u_9 v_5)(u_2 v_4 v_3 u_7 v_9 v_8)(u_3 v_7 v_6 u_8 v_2 v_1)(u_5 u_{10}) \\
\theta_2 &=(u_9 u_1)(u_4 u_6)(u_2 u_8)(u_3 u_7)(v_9 v_1)(v_4 v_6)(v_2 v_8)(v_3 v_7)
\end{align*}

Let $\Gamma_{D6}$ be the embedding in Figure \ref{fig_D6_embed} with $u_{10}$ at $\infty$.  The automorphism $\theta_1$ is induced by an order $6$ glide rotation $h_1$ which rotates $\Gamma_{D6}$ by $\pi/3$ around the vertical axis through $u_5$ and $u_{10}$, while rotating by $\pi$ around a meridian of the hexagonal tube interchanging $u_5$ and $u_{10}$.  The automorphism $\theta_2$ is induced by the involution $h_2$ which rotates $\Gamma_{D6}$ by $\pi$ around an axis through $u_{10}$, $v_{10}$, $u_5$, and $v_5$.  
 
  \begin{figure}[h!]
\begin{center}
\centering\includegraphics[scale=.3]{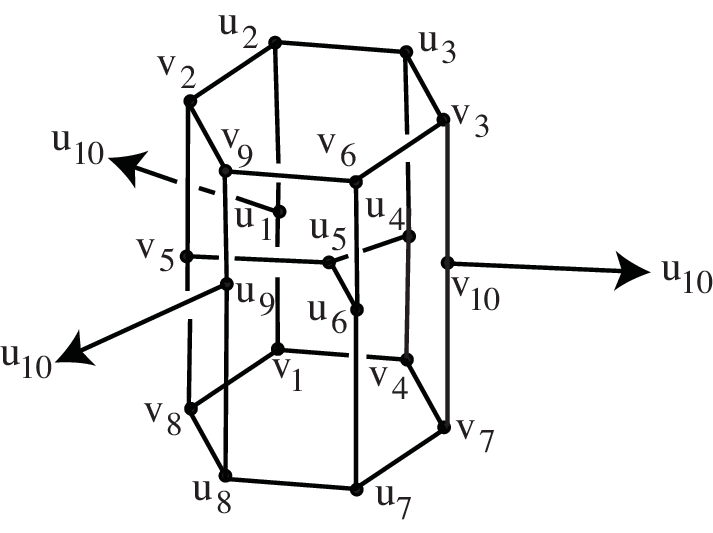}
\end{center}
\caption{ The embedding $\Gamma_{D6}$ of $P(10,3)$ such that $\mathrm{TSG}_+(\Gamma_{D6}')=D_6$, where $v_{10}$ is the point at $\infty$.}
\label{fig_D6_embed}
\end{figure}

Now we create a new embedding $\Gamma_{D6}'$ by adding the $4_1$ knot to the edges of the hexagons at the top and bottom of $\Gamma_{D6}$.  Then $D_6=\langle \theta_1,\theta_2 \rangle \leq \mathrm{TSG}_+(\Gamma_{D6}')$ and every homeomorphism of $(S^3,\Gamma_{D6}')$ leaves this pair of hexagons setwise invariant.   The only positively realizable subgroup of $\mathrm{Aut}(P(10,3))$ containing $D_6$ as a proper subgroup is $A_5 \times \mathbb{Z}_2$.  However, $A_5 \times \mathbb{Z}_2$ contains an element of order $5$, which cannot leave a pair of hexagons setwise invariant.  Thus $D_6=\mathrm{TSG}_+(\Gamma_{D6}')$, and hence $D_6$ is positively realizable.

Finally, observe that the edge $\overline{v_2v_5}$ is not fixed by any non-trivial element of $\mathrm{TSG}_+(\Gamma_{D6}')$.  Thus we can apply Theorem~\ref{SubgroupTheorem} to conclude that every subgroup of $D_6$ is positively realizable for $P(10,3)$.\end{proof}

\begin{proposition}  \label{A4Z2_pos_real}
$A_4 \times \mathbb{Z}_2$ and all of its subgroups are positively realizable and hence realizable for $P(10,3)$.
\end{proposition}
\begin{proof}  Recall that $ \mathrm{TSG}_+(\Lambda)=A_5 \times \mathbb{Z}_2$ is induced by the group $H$ of orientation preserving isometries of $(S^3, \Omega\cup \Omega')$.  We create an embedding $\Lambda_T$ by adding the  $4_1$ knot to the edges $\overline{v_8v_1}$, $\overline{u_9u_{10}}$, $\overline{v_6v_3}$, $\overline{v_2u_2}$, $\overline{u_5u_4}$, $\overline{u_7v_7}$ of $\Lambda$.  Then $ \mathrm{TSG}_+(\Lambda_T)$ is the subgroup of $ \mathrm{TSG}_+(\Lambda)$ taking the set $\{\overline{v_8v_1}, \overline{u_9u_{10}}, \overline{v_6v_3}, \overline{v_2u_2}, \overline{u_5u_4}, \overline{u_7v_7}\}$ to itself.  Hence $ \mathrm{TSG}_+(\Lambda_T)$ is induced by the subgroup $H_T\leq H$ taking the pair of vertices $\{5, 1234\}$ of $\Omega\cup \Omega'$ to itself.  It follows that $H_T$ is the set of orientation preserving isometries of the pair of tetrahedra $\Omega_T$ and $\Omega_T'$ with vertices $1$, $2$, $3$, and $4$, and $1345$, $1245$, $1235$, and $2345$, respectively.  Thus $H_T=A_4\times\mathbb{Z}_2$.
 
 Now observe that if both vertices of the edge $\overline{u_1u_2}$ were fixed by a non-trivial automorphism in $\mathrm{TSG}_+(\Lambda_T)$, then the automorphism would be induced by an element of $H_T$ which interchanges $\overline{u_{2}v_{2}}$ and $\overline{u_{2}u_{3}}$.  But $\overline{u_{2}v_{2}}$ contains a knot and $\overline{u_{2}u_{3}}$ does not.  As this is impossible, we can apply Theorem~\ref{SubgroupTheorem} to conclude that every subgroup of $A_4\times\mathbb{Z}_2$ is positively realizable for $P(10,3)$.  \end{proof}

Next we prove that all of the subgroups of $\mathrm{Aut}(P(10,3))$ that are not positively realizable for $P(10,3)$, are in fact realizable.

\begin{proposition}
 $\mathrm{Aut}(P(10,3))=S_5 \times \mathbb{Z}_2$  and $S_5$ are realizable for $P(10,3)$.
\end{proposition}

\begin{proof} The embedding $\Lambda$, in Figure~\ref{simplex}, is invariant under the group $S_5 \times \mathbb{Z}_2$ of  isometries of the pair of dual $4$-simplices $\Omega$ and $\Omega'$, including a reflection taking each $4$-simplex to itself.   Since $\mathrm{Aut}(P(10,3))=S_5 \times \mathbb{Z}_2$, we have $\mathrm{TSG}(\Lambda)=S_5 \times \mathbb{Z}_2$.

Recall from the proof of Proposition~\ref{A5_edge_embed} that the embedding $\Lambda'$ was obtained from $\Lambda$ by adding the knot $8_{17}$ to every edge of $\Lambda$ oriented from vertices on edges of $\Omega$ to vertices on edges of $\Omega'$.  Then $\Lambda'$ is invariant under the isometries of $(S^3,\Omega)$, and no homeomorphism of $(S^3,\Lambda')$ can interchange $\Omega$ and $\Omega'$.  Thus $\mathrm{TSG}(\Lambda)=S_5$.\end{proof}

\begin{proposition} \label{S4_real}
 $S_4 \times \mathbb{Z}_2$ and $S_4$ are realizable for $P(10,3)$.
\end{proposition}

\begin{proof} Recall from the proof of Proposition~\ref{A4Z2_pos_real} that $\Lambda_{T}$ was obtained from $\Lambda$ by adding a $4_1$ knot to the edges going between the tetrahedron $\Omega_T$ with vertices $1,2,3,4$ and the dual tetrahedron $\Omega'_T$ with vertices $1245, 1345, 1235, 1245$.  Also $\mathrm{TSG}_+(\Lambda_T)$ is induced by the group $A_4 \times \mathbb{Z}_2$ of orientation preserving isometries of $(S^3,\Omega_T\cup\Omega'_T)$.  Now there is a reflection of $(S^3,\Omega_T\cup\Omega'_T)$ which pointwise fixes the sphere containing the vertices $u_9,v_{10},u_1, u_5$  and $u_8,v_7,u_2,u_6$.   Thus the full group of isometries of $(S^3,\Omega_T\cup\Omega'_T)$ including reflections is $S_4 \times \mathbb{Z}_2$.  Since these isometries induce a faithful action on $\Lambda_T$, it follows that $S_4 \times \mathbb{Z}_2\leq \mathrm{TSG}(\Lambda_T)$.

Let $\alpha = (u_1 u_2 u_3 u_4 u_5 u_6 u_7 u_8 u_9  u_{10})(v_1 v_2 v_3 v_4 v_5 v_6 v_7 v_8 v_9 v_{10})$.  Then $\alpha\in \mathrm{Aut}(P(10, 3))$ taking the edge $\overline{v_1v_8}$ to the edge $\overline{v_2v_9}$.  However, in the embedding $\Lambda_T$, the edge $\overline{v_1v_8}$ contains a knot while the edge $\overline{v_2v_9}$ does not.  Thus $\alpha$ is not contained in $\mathrm{TSG}(\Lambda_T)$, and hence $\mathrm{TSG}(\Lambda_T)\not=\mathrm{Aut}(P(10, 3))=S_5 \times \mathbb{Z}_2$.  Since, $S_4 \times \mathbb{Z}_2$ is a maximal subgroup of $S_5 \times \mathbb{Z}_2 $ it follows that $\mathrm{TSG}(\Lambda_T)=S_4 \times \mathbb{Z}_2$.

Finally, we obtain $\Gamma_{S4}$ from $\Lambda_T$ by replacing the invertible $4_1$ knots by non-invertible $8_{17}$ knots which are oriented from $\Omega_T$ to $\Omega_T'$.  Because of the orientation, no homeomorphism of $(S^3,\Gamma_{S4})$ can interchange $\Omega_T$ and $\Omega_T'$.  But all of the other elements of $\mathrm{TSG}(\Lambda_T)$ are also elements of $\mathrm{TSG}(\Gamma_{S4})$. Thus $\mathrm{TSG}(\Gamma_{S4})=S_4$.  \end{proof}

\begin{proposition} \label{D4_real} $D_4 \times \mathbb{Z}_2$ and $\mathbb{Z}_4 \times \mathbb{Z}_2$ are realizable for $P(10,3)$.  
\end{proposition}

\begin{proof}  We again start with the embedding $\Lambda_T$, but now we replace the $4_1$ knots on the edges $\overline{u_7v_7}$ and $\overline{u_2v_2}$ by the achiral and invertible knot $5_2$ to obtain an embedding $\Gamma_{D4}$ such that $\mathrm{TSG}(\Gamma_{D4})$ leaves  $\{\overline{v_7u_7}, \overline{u_2v_2}\}$ and $\{\overline{v_8v_1}, \overline{u_9u_{10}}, \overline{v_6v_3}, \overline{u_5u_4}\}$ setwise invariant.   Then any homeomorphism of $(S^3, \Omega_T \cup \Omega_T')$  which leaves $\Gamma_{D4}$ setwise invariant either interchanges the squares $S=\overline{1234}\subseteq \Omega_T$ and 

\noindent $S'= \overline{1245, 1235, 2345, 1345}\subseteq \Omega_T'$ or leaves each square setwise invariant.

Consider the following automorphisms of $P(10,3)$.
\begin{align*}
\gamma_1 &=(u_5 v_8 u_9 v_6)(u_1 v_{10} u_3 v_4)(u_7 v_2)(v_5 u_8 v_9 u_6)(v_1 u_{10} v_3 u_4)(u_2 v_7) \\
\gamma_2 &=(u_1 v_{10})(u_2 v_7)(u_3 v_4)(u_6 v_5)(u_7 v_2)(u_8 v_9)(v_1 v_3)(v_6 v_8) \\
\gamma_3 &=(v_3 v_8)(v_1 v_6)(v_9 v_4)(v_7 v_2)(v_5 v_{10})(u_3 u_8)(u_1 u_6)(u_9 u_4)(u_7 u_2)(u_5 u_{10})
\end{align*}

\smallskip

Since $\Omega_T$ and $\Omega'_T$ are dual regular tetrahedra in $S^3$, there is a reflection composed with an order $4$ rotation $g$ of $(S^3, \Omega_T, \Omega_T')$ which rotates $S$  and $S'$ interchanging the edges $\overline{13}$ and $\overline{24}$ and interchanging the edges $\overline{245}$ and $\overline{135}$.   Also, there is an order $2$ rotation $f$ of $(S^3, \Omega_T, \Omega_T')$, which turns over the square $S$ interchanging the edges $\overline{14}$ and $\overline{23}$, and turns over the square $S'$ interchanging the edges $\overline{145}$ and $\overline{235}$.  Now, $g$ and $f$ leave $\Gamma_{D4}$ setwise invariant inducing $\gamma_1$ and $\gamma_2$, respectively. Since $g$ rotates $S_1$ and $S_2$ and $f$ turns $S_1$ and $S_2$ over, $\langle g, f\rangle=D_4$.  Finally, since the induced action on $\Gamma_{D4}$ is faithful, $\langle \gamma_1, \gamma_2\rangle=D_4$

Recall from the proof of Proposition~\ref{A5_edge_embed} that there is an order $2$ glide rotation $h$ of $S^3$ which interchanges $\Omega$ and $\Omega'$, interchanging each vertex $i\in\Omega$ with the vertex $jklm\in\Omega'$ where  $i\not\in \{j,k,l,m\}$. Observe that $h$ interchanges $S$ and $S'$ and leaves $\Gamma_{D4}$ setwise invariant inducing $\gamma_3$.  Also, $\gamma_3$ commutes with $\gamma_1$ and $\gamma_2$, and hence $\langle\gamma_1,\gamma_2,\gamma_3\rangle= D_4\times \mathbb{Z}_2\leq \mathrm{TSG}(\Gamma_{D4})$.  However, since $\mathrm{TSG}(\Gamma_{D4})$ is a proper subgroup of $\mathrm{TSG}(\Lambda_{T})=S_4 \times \mathbb{Z}_2$, and $D_4\times \mathbb{Z}_2$ is maximal in $S_4 \times \mathbb{Z}_2$, it follows that $\mathrm{TSG}(\Gamma_{D4})=D_4\times \mathbb{Z}_2$.

Next, we replace a neighborhood of $u_4$ by one where each edge incident to $u_4$ is knotted around the next such edge as illustrated in Figure~\ref{Fig:orient}, except that now the chiral $3_1$ knot in the figure is replaced by the achiral $4_1$ knot.  As explained in the proof of Proposition~\ref{SL2}, this gives an order to the edges around $u_4$.  We then apply the isometry group $\langle g,h\rangle$ to the neighborhood $N(u_4)$, so that the edges in the orbit of $N(u_4)$ also have an order.  Since $g$ rotates $S_1$ and $S_2$, and $h$ is an involution interchanging $S_1$ and $S_2$, no vertex in the orbit of $u_4$ is fixed by a non-trivial element of $\langle g,h\rangle$.  Thus, replacing the original neighborhoods of these vertices by the new ones yields a well-defined embedding $\Gamma_{Z4}$ of $P(10,3)$ such that $\mathbb{Z}_4\times \mathbb{Z}_2=\langle\gamma_1,\gamma_3\rangle\leq \mathrm{TSG}(\Gamma_{Z4})$.  

Now suppose that a homeomorphism $f'$ of $(S^3,\Gamma_{Z4})$ induces $\gamma_2$.  Then $f'$ fixes the edge $\overline{u_5u_4}$ and interchanges $\overline{u_4u_3}$ and $\overline{u_4v_4}$.  But this means that $f'$ reverses the order of the edges incident to $u_4$, which is impossible since each of these edges is knotted around the next one.  Thus $\mathrm{TSG}(\Gamma_{Z4})$ is a proper subgroup of $\mathrm{TSG}(\Gamma_{D4})$.  Since $\mathbb{Z}_4\times \mathbb{Z}_2$ is maximal in $D_4\times \mathbb{Z}_2$, we must have $\mathrm{TSG}(\Gamma_{Z4})=\mathbb{Z}_4\times \mathbb{Z}_2$.\end{proof}

\begin{proposition}  \label{D6Z2_real}
 $D_6 \times \mathbb{Z}_2$ and $\mathbb{Z}_6 \times \mathbb{Z}_2$ are realizable for $P(10,3)$.
\end{proposition}
\begin{proof}
We begin with the embedding $\Gamma_{D_6}'$ and automorphisms $\theta_1$ and $\theta_2$ from the proof of Proposition \ref{P103_D6_pos} such that $\mathrm{TSG}_+(\Gamma_{D_6}')=\langle \theta_1,\theta_2\rangle=D_6$.  Now let $$\theta_3= (u_3 v_4)(v_3 v_7)(v_6 u_7)(v_9 u_8)(v_2 v_8)(u_2 v_1)(u_3 v_4).$$ Then $\theta_3$ commutes with $\theta_1$ and $\theta_2$, and hence $\langle\theta_1,\theta_2,\theta_3\rangle=D_6\times\mathbb{Z}_2$.  Also, $\theta_3$ is induced on $\Gamma_{D_6}'$ by a reflection through the sphere containing the set of vertices $\{u_{10}, u_5, u_4, u_6, v_5, u_9, u_1\}$.  Thus $D_6\times\mathbb{Z}_2\leq \mathrm{TSG}(\Gamma_{D_6}')$.  

Because of the $4_1$ knots in the edges of the top and bottom hexagons, every homeomorphism of $(S^3,\Gamma_{D6}')$ takes this pair of hexagons to itself.  Thus $\mathrm{TSG}(\Gamma_{D_6}')$ is a proper subgroup of $S_5 \times \mathbb{Z}_2 = \mathrm{Aut}(P(10,3))$.  However, since $D_6 \times \mathbb{Z}_2$ is a maximal subgroup of $S_5 \times \mathbb{Z}_2 $, we have $\mathrm{TSG}(\Gamma_{D_6}')=D_6\times\mathbb{Z}_2$.

Next we replace the $4_1$ knots in the edges of the top and bottom hexagons of $\Gamma_{D6}' $ by the non-invertible knot $8_{17}$ oriented consistently around the two hexagons to get a new embedding $\Gamma_{Z6}$.  Now no homeomorphism of $(S^3,\Gamma_{Z6})$ can turn either hexagon over.  Hence the automorphism $\theta_2$ is not in $\mathrm{TSG}(\Gamma_{Z_6})$, but all of the other elements of $\mathrm{TSG}(\Gamma_{D_6}')$ are in $\mathrm{TSG}(\Gamma_{Z_6})$.  Thus $\mathbb{Z}_6 \times \mathbb{Z}_2\leq \mathrm{TSG}(\Gamma_{Z_6})$.  Now since $\mathbb{Z}_6 \times \mathbb{Z}_2$ is a maximal subgroup of $D_6\times\mathbb{Z}_2$, we have $\mathbb{Z}_6 \times \mathbb{Z}_2= \mathrm{TSG}(\Gamma_{Z_6})$. \end{proof}

Since $3^2\equiv -1\pmod{10}$, by Theorem~\ref{TSG(k2=-1)} we have the following.

\begin{proposition} $\mathbb{Z}_{10} \rtimes \mathbb{Z}_4$ and all of its subgroups are realizable for $P(10,3)$.
\end{proposition}

The following is a summary of our results for $P(10,3)$.

\begin{theorem}
$\mathrm{Aut}(P(10,3))= S_5 \times \mathbb{Z}_2$ and all of its subgroups are realizable for $P(10,3)$.  Futhermore:
\begin{enumerate}

\item   $\mathbb{Z}_5\rtimes \mathbb{Z}_4$, $\mathbb{Z}_{10} \rtimes \mathbb{Z}_4$, $S_5 \times \mathbb{Z}_2$, $\mathbb{Z}_6 \times \mathbb{Z}_2$, $D_6 \times \mathbb{Z}_2$, $S_4$, $S_5$, $S_4 \times \mathbb{Z}_2$, $\mathbb{Z}_4 \times \mathbb{Z}_2$, and $D_4 \times \mathbb{Z}_2$ are not positively realizable.

\item $D_{10}$, $\mathbb{Z}_{10}$, $D_5$, $\mathbb{Z}_5$, $D_2$, $\mathbb{Z}_2$, $A_5 \times \mathbb{Z}_2$, $A_5$, $A_4\times \mathbb{Z}_2$, $A_4$, $\mathbb{Z}_3$, $\mathbb{Z}_2\times\mathbb{Z}_2\times\mathbb{Z}_2$, $D_4$, $\mathbb{Z}_4$, $D_6$, $\mathbb{Z}_6$, and $D_3$  are positively realizable.
\end{enumerate}
\end{theorem}

\section{Conclusion}

In the table below we summarize our results for $P(n,k)$.  A checkmark in the ``$\mathrm{TSG}(\Gamma)$ complete'' column means that every subgroup of the automorphism group is realizable.  A checkmark $\checkmark$ in the ``$\mathrm{TSG}_+(\Gamma)$ complete'' column means that every subgroup of the automorphism group is positively realizable.  An X means that some subgroups are not realizable or not positively realizable.   The first three rows are in the non-exceptional cases.  The final five rows are all the exceptional cases except for $P(12,5)$ and $P(24,5)$, which will be analyzed in a subsequent paper. 

\begin{table}
\begin{center}
\begin{tabular}{| c | c |  c | c | } \hline
& & & \\
$P(n,k)$ & $\mathrm{Aut}(P(n,k))$ &  $ \mathrm{TSG}_+(\Gamma)$ complete &  $\mathrm{TSG}(\Gamma)$ complete \\ 
& & & \\ \hline
$k^2 \neq \pm 1 \ \mathrm{ mod } \ n$ & $D_n$ & $\checkmark$ & $\checkmark$ \\ \hline
$k^2 =  1 \ \mathrm{ mod } \ n$ & $D_n \rtimes \mathbb{Z}_2$ & $\checkmark$ & $\checkmark$ \\ \hline
$k^2 =  -1 \ \mathrm{ mod } \ n$ & $\mathbb{Z}_n \rtimes \mathbb{Z}_4$ & X & $\checkmark$ \\ \hline
$P(4,1)$ & $S_4 \times \mathbb{Z}_2$ & $\checkmark$ & $\checkmark$ \\ \hline
$P(5,2)$ & $S_5$ & X & X  \\ \hline
$P(8,3)$ & $GL(2,3) \rtimes \mathbb{Z}_2$ & $\checkmark$ & $\checkmark$ \\ \hline
$P(10,2)$ & $A_5 \times \mathbb{Z}_2$  & $\checkmark$ & $\checkmark$ \\ \hline
$P(10,3)$ & $S_5 \times \mathbb{Z}_2$  & X & $\checkmark$ \\ \hline
\end{tabular}
\end{center}\caption{This table summarizes our results.  \emph{Complete}
 means that all of the subgroups are realized.  The first three rows are in the non-exceptional cases}
\end{table}

\section*{Acknowledgements} We thank AIM for its support of the REUF program which enabled us to start this project and supported our meeting at the JMM.  We also thank ICERM for hosting our second meeting.  Finally, we thank the referee for helpful suggestions.

\bigskip

\bibliographystyle{plain}
\bibliography{TSG_GenPet}

\end{document}